\documentclass[10pt]{amsart}

\usepackage{verbatim,amsopn,amsthm,amssymb,amsmath,graphicx,subfigure}
\usepackage{hyperref}

\theoremstyle{plain}
\newtheorem{thm}{Theorem}[section]
\newtheorem{lem}[thm]{Lemma}
\newtheorem{prop}[thm]{Proposition}
\newtheorem{cor}[thm]{Corollary}
\newtheorem*{siegelslemma}{Siegel's Lemma}

\newtheorem*{ThmPolyUpper}{Theorem A}
\newtheorem*{ThmFreeAbComplexity}{Theorem B}
\newtheorem*{ThmEOC}{Theorem C}
\newtheorem*{ThmIEOC}{Theorem D}

\theoremstyle{definition}
\newtheorem{defn}[thm]{Definition}

\theoremstyle{remark}
\newtheorem*{remark}{Remark}

\DeclareMathOperator{\Cay}{Cayley}
\DeclareMathOperator{\len}{len}

\DeclareMathOperator{\id}{Id}

\newcommand{\xup}{{X \cup \mathcal{P}}}
\newcommand{\ds}{\displaystyle}

\begin{document}

\title{Quantifying the Residual Properties of $\Gamma$-Limit Groups}
\author{Brent B. Solie}
\address{Knox College, 2 East South Street, Galesburg, IL 61401}
\email{b.b.solie@gmail.com}
\date{\today}
\subjclass[2010]{Primary 20P05}
\keywords{geometric group theory, residual hyperbolicity, relative hyperbolicity, limit group}

\begin{abstract}
  \noindent  Let $\Gamma$ be a fixed hyperbolic group.
  The \emph{$\Gamma$-limit groups} of Sela are exactly the finitely generated, fully residually $\Gamma$ groups.
  We give a new invariant of $\Gamma$-limit groups called $\Gamma$-discriminating complexity and show that the $\Gamma$-discriminating complexity of any $\Gamma$-limit group is asymptotically dominated by a polynomial.
  Our proof relies on an embedding theorem of Kharlampovich-Myasnikov which states that a $\Gamma$-limit group embeds in an iterated extension of centralizers over $\Gamma$.
  The result then follows from our proof that if $G$ is an iterated extension of centralizers over $\Gamma$, the $G$-discriminating complexity of a rank $n$ extension of a cyclic centralizer of $G$ is asymptotically dominated by a polynomial of degree $n$.
\end{abstract}

\maketitle

\section{Introduction}

Quantitative analysis of group properties is an increasingly active field in modern group theory.
In particular, the various residual properties of groups have proven themselves quite suitable for investigation through quantitative means.

Let $P$ be a property of groups, and recall that a group $G$ is \emph{residually $P$} \index{residually $P$} if for every nontrivial element $g \in G$, there is a homomorphism $\phi: G \rightarrow H$ such that $H$ is a group with property $P$ and $\phi(g) \neq 1$.
We say that a group is \emph{fully residually $P$} \index{fully residually $P$} if for every finite subset of nontrivial elements $S \subseteq G-1$, there is a homomorphism $\phi: G \rightarrow H$ such that $H$ is a group with property $P$ and $1 \notin \phi(S)$.

(An alternate definition of fully residually $P$ insists that the homomorphism $\phi$ not just avoid $1$ but actually be injective on $S$.
Note that $\phi$ is injective on $S$ if and only if the image under $\phi$ of the set $\{uv^{-1} : u,v \in S, u \neq v\}$ does not include 1, so these definitions are equivalent.
Also note that we also do not require our homomorphisms to be surjective, as may sometimes be the case when discussing residual properties.)

For instance, let $G$ be a residually finite group with finite generating set $X$.
Let $f: \mathbb{N} \rightarrow \mathbb{N}$ be such that whenever $g \in G-1$ has $X$-length at most $R$, then there exists $\phi:G \rightarrow H$ such that $\phi(g) \neq 1$ and $|H| \leq f(R)$.
When $f$ is the smallest such function, then we think of $f$ as measuring the complexity of the residual finiteness of $G$; we may also think of $f$ as measuring the growth of the number of subgroups of $G$ with respect to index.
This version of complexity has been studied extensively by Bou-Rabee in \cite{Bou-Rabee2010}, with additional results by Kassabov and Matucci \cite{Kassabov2009}.

Bou-Rabee has obtained further results by restricting his attention to finite nilpotent or finite solvable quotients.
This yields group invariants known as the \emph{nilpotent Farb growth} and the \emph{solvable Farb growth}, and Bou-Rabee has obtained new characterizations of algebraic group properties in terms of the asymptotic properties of these growth functions.
For instance, Bou-Rabee has shown that a finitely generated group $G$ is nilpotent if and only if it has nilpotent Farb growth which is polynomial in $\log(n)$ \cite{Bou-Rabee2010}.
Similarly, a finitely generated group is solvable and virtually nilpotent if and only if it has solvable Farb growth that is polynomial in $\log(n)$ \cite{Bou-Rabee2011}.

Rather than considering residually finite groups, we will study another well-known class of groups with strong residual properties: the $\Gamma$-limit groups of Sela.
Let $\Gamma$ be a fixed torsion-free hyperbolic group.
A \emph{$\Gamma$-limit group} $G$ is a finitely generated, fully residually $\Gamma$ group: for any finite subset $S \subseteq G-1$, there exists a homomorphism $\phi: G \rightarrow \Gamma$ such that $1 \notin \phi(S)$.
We say that the set $S$ is \emph{$\Gamma$-discriminated} by $\phi$.

Fix finite generating sets $X$ and $Y$ for $G$ and $\Gamma$, respectively.
Let the homomorphism $\phi_R: G \rightarrow \Gamma$ discriminate $B_R(G,X)-1$, where $B_R(G,X)$ is the closed ball of radius $R$ in $G$ with respect to $X$.
Here, we measure the complexity of $\phi_R$ by the maximum $Y$-length over all images of elements of $X$.
The minimum complexity required to discriminate each set $B_R(G,X)-1$, as a function of $R$, is called the $\Gamma$-discriminating complexity of $G$, and it is an invariant of $G$ up to asymptotic equivalence. (See Definition \ref{Defn--Discriminating complexity}.)

Our main result on the $\Gamma$-discriminating complexity of $\Gamma$-limit groups is the following:

\begin{ThmPolyUpper}[c.f. Theorem \ref{Solie--Theorem--limit group has poly upper}]
  The $\Gamma$-discriminating complexity of a $\Gamma$-limit group is asymptotically dominated by a polynomial.
\end{ThmPolyUpper}

In order to prove Theorem A, we must first start with the simplest examples of $\Gamma$-limit groups: the finitely generated, free Abelian groups.
The free Abelian group $\mathbb{Z}^n$ is fully residually $\mathbb{Z}$, and our next main result establishes its $\mathbb{Z}$-discriminating complexity.

\begin{ThmFreeAbComplexity}[c.f. Theorem \ref{Solie--Theorem--Free Abelian rank n has complexity n-1}]
  The $\mathbb{Z}$-discriminating complexity of $\mathbb{Z}^n$ is asymptotically equivalent to a polynomial of rank $n-1$.
\end{ThmFreeAbComplexity}

The fundamental construction in our study of $\Gamma$-limit groups is the extension of a centralizer, a construction in which free Abelian groups play a central role.
Informally, if $G$ is a $\Gamma$-limit group, we may construct another $\Gamma$-limit group $G'$ by extending a centralizer of $G$ by a free Abelian group of finite rank. (See Definition \ref{Defn--Extension of a centralizer}.)

Our main technical lemma, Lemma \ref{Solie--Lemma--Padding for more general words}, is motivated by the well-known ``big powers'' property of hyperbolic groups.
If $\Gamma$ is a hyperbolic group and $u \in \Gamma$ generates its own centralizer, then for any tuple of elements $(g_1, g_2, \dots, g_k)$ of elements of $G - \langle u \rangle$, there is an integer $N$ such that
\begin{equation*}
  u^{n_0} g_1 u^{n_1} g_2 u^{n_2} \dots u^{n_{k-1}} g_k u^{n_k}
\end{equation*}
is nontrivial in $\Gamma$ whenever $|n_i|>N$ for $i=1, \dots, k-1$ and either $|n_i|>N$ or $n_i=0$ for $i=0, k$.

The big powers property seems to appear first due to B. Baumslag in his study of fully residually free groups \cite{Baumslag1967}; a later version appears due to Ol'shanski\u{\i} in the context of hyperbolic groups \cite{Olcprimeshanskiui1993}.
Most recently, the big powers property is proven by Kharlampovich and Myasnikov for relatively hyperbolic groups in \cite{Kharlampovich2009} using the techniques of Osin from \cite{Osin2005,Osin2006}.
Lemma \ref{Solie--Lemma--Padding for more general words} is an analysis of the big powers property for relatively hyperbolic groups with the goal of analyzing the dependence of $N$ on the group $G$, generating set $X$, and the elements $g_i$ and $u$.

By iterating the extension of centralizer construction, we obtain a group known as an \emph{iterated extension of centralizers} (see Definition \ref{Defn--Iterated extension of centralizers}).
Iterated extensions of centralizers are relatively hyperbolic and therefore have the big powers property.
By combining Theorem B with our analysis of the big powers property, we obtain our third main result.

\begin{ThmEOC}[c.f. Theorem \ref{Solie--Theorem--EOC has poly upper}]
  Let $G$ be an iterated extension of centralizers over $\Gamma$.
  Let $G'$ be a rank $n$ extension of a cyclic centralizer of $G$.
  Then the $G$-discriminating complexity of $G'$ is asymptotically dominated by a polynomial of degree $n$.
\end{ThmEOC}

Repeated application of Theorem C gives us our final main result, a bound on the discriminating complexity of an arbitrary iterated extension of centralizers over $\Gamma$.

\begin{ThmIEOC}[c.f. Theorem \ref{Solie--Theorem--IEOC has poly upper}]
  The $\Gamma$-discriminating complexity of an iterated extension of centralizers over $\Gamma$ is asymptotically dominated by a polynomial with degree equal to the product of the ranks of the extensions.
\end{ThmIEOC}

Theorem D then directly implies Theorem A via a theorem of Kharlampovich and Myasnikov, which states that every $\Gamma$-limit group embeds in some iterated extension of centralizers over $\Gamma$ \cite{Kharlampovich2009}. 

\section{Background}

Let $G$ be a group with a generating set $X$.

\begin{defn}[Cayley graph] \index{Cayley graph}
  \label{Defn--Cayley graph}
  The \emph{Cayley graph} of $G$ with respect to the generating set $X$, denoted $\Cay(G,X)$, is an oriented graph with vertex set in bijection with $G$.
  The edge set is in bijection with $G \times X$, where the pair $(g,x)$ corresponds to an edge having initial vertex $g$, terminal vertex $gx$, and label $x$.
\end{defn}

For a fixed set $X$, an \emph{$X$-word} is a finite sequence of elements of $X$.
By $X^*$ we denote the set of all $X$-words, including the empty word.
When $X$ is a generating set for a group $G$, then every element of $X^*$ represents an element of $G$.
Where it is necessary to distinguish between them, we will denote by $\overline{w}$ the element of $G$ represented by $w \in X^*$.

Recall that for an element $g \in G$, the \emph{word length with respect to $X$} or \emph{$X$-length}, of $g$, denoted $|g|_X$, is number of letters in the shortest $X$-word representing $g$.
Equivalently, $|g|_X$ is the number of edges in the shortest path from $1$ to $g$ in $\Cay(G,X)$.

For an integer $R \geq 0$, the \emph{ball of radius $R$ with respect to generating set $X$} is the set $B_R(G,X) = \{ g \in G : |g|_X \leq R \}$.
Where $G$ and $X$ are clear from context, we will denote this set simply by $B_R$.
Note that when $X$ is a finite set, then $B_R$ is also finite for any integer $R \geq 0$.

Finally, for elements $g,h \in G$, the \emph{right-conjugate of $h$ by $g$} is the element $h^g := g^{-1} h g$. 

\subsection{$\Gamma$-Limit Groups}

Sela first introduced the notion of a limit group in \cite{Sela2001} in his investigation of groups having the elementary theory of a non-Abelian free group.
Sela later generalized this notion to that of a $\Gamma$-limit group, where $\Gamma$ is some fixed torsion-free hyperbolic group \cite{Sela2009}.

\begin{defn}[Residual properties] \index{residual property} \index{discriminating homomorphism}
  \label{Defn--Residual properties}
  Fix a group $H$. We say that a group $G$ is \emph{residually $H$} if for any $g \in G-1$, there exists a homomorphism $\phi_g: G \rightarrow H$ such that $\phi_g(g) \neq 1$.
  A group $G$ is \emph{fully residually $H$} if for any finite set $S$ of nontrivial elements of $G$, there exists a homomorphism $\phi_S: G \rightarrow H$ such that $1 \notin \phi_S(S)$.
  The homomorphisms $\phi_g$ and $\phi_S$ are called \emph{$H$-discriminating homomorphisms} for $g$ and $S$, respectively.
\end{defn}

For the remainder of this chapter, $\Gamma$ will denote a non-Abelian, torsion-free hyperbolic group.

\begin{defn}[$\Gamma$-limit group \cite{Sela2009}] \index{$\Gamma$-limit group}
  \label{Defn--Gamma limit group}
  We say that a group $G$ is a \emph{$\Gamma$-limit group} if $G$ is finitely generated and fully residually $\Gamma$.
\end{defn}

A trivial example of a $\Gamma$-limit group is $\Gamma$ itself.
For a more complicated example, it is well-known that fundamental groups of closed, orientable hyperbolic surfaces are $F_2$-limit groups, where $F_2$ denotes the free group of rank two.

We may produce new $\Gamma$-limit groups from existing limit groups through a construction called an \emph{extension of a centralizer}.
Extensions of centralizers will provide the basis for our analysis of the residual properties of limit groups.

Let $G$ be a group, and given $g \in G$, let $C_G(u) = \{g \in G : u^g = u\}$ denote the centralizer of $u$ in $G$.

\begin{defn}[Extension of a centralizer \cite{Myasnikov1996}] \index{extension of a centralizer}
  \label{Defn--Extension of a centralizer}
  Suppose that for some $u \in G$, the centralizer $C=C_G(u)$ is Abelian and that $\phi: C \rightarrow A$ is injective for some Abelian group $A$. We call the amalgamated product
  \begin{displaymath}
    G(u,A) := \displaystyle{G *_{C = \phi(C)} A}
  \end{displaymath}
  the \emph{extension of the centralizer $C$ by $A$ with respect to $\phi$}.
  We will call the extension \emph{direct} if $A = \phi(C) \times B$ for some subgroup $B \leq A$.
  A direct extension is \emph{free of rank $n$} if $B \cong \mathbb{Z}^n$.
\end{defn}

Having given the most general definition, we will now assume that all extensions of centralizers are free and of finite rank.
We will omit reference to the homomorphism $\phi$ when it is clear from context.

The following proposition is well-known and will serve as the starting point for our investigation of the residual properties of $\Gamma$-limit groups.

\begin{prop}
  The extension of centralizer $G(u,A)$ is a $G$-limit group.
\end{prop}

\begin{prop}[{\cite[Corollary 3]{Myasnikov1996}}]
  \label{Myasnikov-Remeslennikov--Prop--Maximal Abelian subgroups}
  A maximal Abelian subgroup of $G(u,A)$ is either conjugate to a subgroup of $G$, conjugate to $A$, or cyclic.
\end{prop}

\begin{defn}[Iterated extension of centralizers] \index{iterated extension of centralizers}
  \label{Defn--Iterated extension of centralizers}
  Let $G$ be a group. An \emph{iterated extension of centralizers over $G$} is a group $H$ for which there exists a finite series
  \begin{displaymath}
    G= G_0 \leq G_1 \leq \dots \leq G_k = H
  \end{displaymath}
  such that for $i=0, \dots, k-1$, each $G_{i+1}$ is an extension of a centralizer of $G_i$.
\end{defn}

Since each $G_{i+1}$ is fully residually $G_i$, we immediately obtain the following:

\begin{prop}
  An iterated extension of centralizers over $G$ is fully residually $G$.
\end{prop}

The following theorem of Kharlampovich and Myasnikov will allow us to approach the residual properties of arbitrary $\Gamma$-limit groups by considering iterated extensions of centralizers.

\begin{prop}[{\cite[Theorems D, E]{Kharlampovich2009}}]
  \label{Kharlampovich-Myasnikov--Prop--Embedding theorem for limit groups}
  Every $\Gamma$-limit group embeds into some iterated extension of centralizers over $\Gamma$.
\end{prop}

Recall that a subgroup $H \leq G$ is \emph{malnormal} if $H \cap H^g = 1$ for all $g \in G-H$.

\begin{defn}[CSA group \cite{Myasnikov1996}] \index{CSA group}
  \label{Defn--CSA}
  A group $G$ is called a \emph{CSA-group} if every maximal Abelian subgroup of $G$ is malnormal. $G$ is called a \emph{CSA*-group} if it is a CSA-group and has no elements of order 2.
\end{defn}

We summarize some of the important properties of CSA- and CSA*-groups.

\begin{prop}[\cite{Myasnikov1996}]
  \label{Myasnikov-Remeslennikov--Prop--CSA Properties}
  \mbox{}
  \begin{enumerate}
    \item
      Any torsion-free hyperbolic group is a CSA*-group.
    \item
      The class of CSA*-groups is closed under iterated extensions of centralizers.
    \item
      Let $G$ be a CSA-group and let $A \leq G$ be a maximal Abelian subgroup.
      Then there is $u \in G$ for which $A = C_G(u)$.
    \item
      Let $G$ be a CSA-group. For any maximal Abelian subgroup $A$, $N_G(A) = A$.
    \item
      Let $G$ be a CSA-group. Then commutativity is a transitive relation on the set $G-1$.
  \end{enumerate}
\end{prop}

\subsection{Relative hyperbolicity}

The following discussion is taken from Osin \cite{Osin2006} with some minor modifications to notation inspired by Hruska \cite{Hruska2010}.

By a pair $(G, \mathbb{P})$ we denote a group $G$ with a distinguished set of subgroups $\mathbb{P} = \{P_\lambda\}_{\lambda \in \Lambda}$.
A subgroup $H \leq G$ is called \emph{parabolic} if it is conjugate into some $P \in \mathbb{P}$, and \emph{hyperbolic} otherwise.
We call the conjugates of the elements of $\mathbb{P}$ \emph{maximal parabolic subgroups}.

\begin{defn}[Relative generating set] \index{relative generating set}
  \label{Defn--Relative generating set}
  Let $\mathcal{P} = \displaystyle{\bigcup_{\lambda \in \Lambda} (P_\lambda-\{1\})}$. We say that $X \subseteq G$ is a \emph{relative generating set for $(G, \mathbb{P})$} if $G$ is generated by $X \cup \mathcal{P}$. If $X$ is finite, we call it a \emph{finite relative generating set}.
\end{defn}

\begin{defn}[Relative presentation] \index{relative presentation}
  \label{Defn--Relative presentation}
  We may consider $G$ as a quotient of the group
  \begin{displaymath}
    F := \displaystyle{( \ast_{\lambda \in \Lambda} P_\lambda )} * F(X),
  \end{displaymath}
  where $F(X)$ is the free group with basis $X$.
  Note that the group $F$ is generated by $\xup$.

  For each $\lambda \in \Lambda$, let $S_\lambda$ denote all the words in $(P_\lambda - 1)^*$ which represent the identity in $P_\lambda$.
  Further denote
  \begin{displaymath}
    \mathcal{S} := \displaystyle{\bigcup_{\lambda \in \Lambda} S_\lambda}.
  \end{displaymath}

  Let $\mathcal{R} \subseteq (\xup)^*$ be such that the normal closure of $\mathcal{R}$ generates the kernel of the homomorphism $F \rightarrow G$.
  We say that $(G, \mathbb{P})$ has the \emph{relative presentation}
  \begin{equation}
    \label{Eqn--Relative presentation}
    \langle X, \mathcal{P} \: | \: \mathcal{R}, \mathcal{S} \rangle.
  \end{equation}
  If $X$ and $\mathcal{R}$ are finite, then we say that the relative presentation \eqref{Eqn--Relative presentation} is \emph{finite}.
  If $(G, \mathbb{P})$ has a finite relative presentation, we say that $(G,\mathbb{P})$ is \emph{finitely relatively presented}.
\end{defn}

Suppose that $(G,\mathbb{P})$ has a relative presentation as in \eqref{Eqn--Relative presentation}.
If $W \in (\xup)^*$ represents the identity in $G$, then there is an expression
\begin{equation}
  \label{Eqn--Product of conjugates}
  W =_F \displaystyle{\prod_{i=1}^k R_i^{f_i}}
\end{equation}
with equality in the group $F$ and such that $R_i \in \mathcal{R}$ and $f_i \in F$ for each $i$.

\begin{defn}[Relative isoperimetric function] \index{relative isoperimetric function}
  \label{Defn--Relative isoperimetric function}
  Let $\theta: \mathbb{N} \rightarrow \mathbb{N}$.
  We say that $\theta$ is a \emph{relative isoperimetric function} for $(G, \mathbb{P})$ if there exists a finite relative presentation with $X$ and $\mathcal{R}$ as above such that for any $W \in (\xup)^*$ with $|W|_\xup \leq n$, there exists an expression of the form \eqref{Eqn--Product of conjugates} such that $k \leq \theta(n)$.
\end{defn}

\begin{defn}[Relative Dehn function] \index{relative Dehn function}
  \label{Defn--Relative Dehn function}
  We call the smallest relative isoperimetric function for a relative presentation the \emph{relative Dehn function} of that relative presentation.
  If a relative presentation has no finite relative isoperimetric function, then we say that the relative Dehn function for that relative presentation is not well-defined.
\end{defn}

\begin{defn}[Relatively hyperbolic group] \index{relatively hyperbolic group}
  \label{Defn--Relatively hyperbolic group}
  We say that $(G, \mathbb{P})$ is a \emph{relatively hyperbolic group} if $(G, \mathbb{P})$ has a finite relative presentation with a well-defined, linear relative Dehn function.
\end{defn}

We will now fix a non-Abelian, torsion-free hyperbolic group $\Gamma$.
Our goal is next to show that an iterated extension of centralizers over $\Gamma$ is hyperbolic relative to its maximal non-cyclic Abelian subgroups.
We begin by noting the following results which may both be found in \cite{Dahmani2003}.

\begin{prop}[\cite{Dahmani2003}]
  \label{Dahmani--Prop--Hyperbolic cyclic subgroups can be considered parabolic}
  Let $(G, \mathbb{P})$ be a torsion-free relatively hyperbolic group. Let $U$ be a cyclic hyperbolic subgroup such that $N_G(U)=U$.
  Then $(G, \mathbb{P}\cup\{U\})$ is also a torsion-free relatively hyperbolic group.
\end{prop}

\begin{prop}[\cite{Dahmani2003}]
  \label{Dahmani--Prop--Combination Theorem for Relatively Hyperbolic Groups}
  Let $(G_1, \mathbb{P}_1)$ and $(G_2, \mathbb{P}_2)$ be relatively hyperbolic groups.
  Let $P \in \mathbb{P}_1$, and suppose that $P$ is isomorphic to a parabolic subgroup of $(G_2, \mathbb{P}_2)$.
  Let $G = G_1 *_P G_2$.
  Then $(G, (\mathbb{P}_1-\{P\})\cup\mathbb{P}_2))$ is relatively hyperbolic.
\end{prop}

\begin{cor}
  \label{Solie--Cor--Gamma limit groups are relatively hyperbolic}
  An iterated extension of centralizers over a torsion-free hyperbolic group $\Gamma$ is hyperbolic relative a set of representatives of conjugacy classes of maximal non-cyclic Abelian subgroups.
\end{cor}

\begin{proof}
  We induct on $k$, the number of steps in the iterated extension. If $k=0$, $G_k = \Gamma$ is hyperbolic and we are done.

  Suppose that $(G_k, \mathbb{P}_k)$ is relatively hyperbolic, where $\mathbb{P}_k$ is a set of representatives of conjugacy classes of maximal non-cyclic Abelian subgroups of $G_k$.
  Without loss of generality, we may assume that $G_{k+1}$ is constructed by extending the centralizer $C(u)=C_{G_{k}}(u)$ of a hyperbolic element $u \in G_k$ by a rank $n$ free Abelian group $A$, so that
  \begin{displaymath}
    G_{k+1} = G_k *_{C(u)} A.
  \end{displaymath}

  Since $u$ is hyperbolic in the CSA-group $(G_k, \mathbb{P}_k)$, the centralizer $C(u)$ is maximal Abelian and $N_{G_k}(C(u)) = C(u)$ by Proposition \ref{Myasnikov-Remeslennikov--Prop--CSA Properties}.
  Moreover, $C(u)$ is cyclic; otherwise, $u$ would be contained in a maximal non-cyclic Abelian subgroup of $(G_k, \mathbb{P}_k)$, contradicting that $u$ is hyperbolic.
  Therefore, by Proposition \ref{Dahmani--Prop--Hyperbolic cyclic subgroups can be considered parabolic},  $(G_k, \mathbb{P}_k \cup \{C(u)\})$ is relatively hyperbolic.
  The free Abelian group $A$ may be viewed as the relatively hyperbolic group $(A, \{A\})$, so $C(u) \leq A$ is parabolic.
  By Proposition \ref{Dahmani--Prop--Combination Theorem for Relatively Hyperbolic Groups}, $(G_{k+1}, \mathbb{P}_k \cup \{A\})$ is therefore a relatively hyperbolic group.
  Finally, Proposition \ref{Myasnikov-Remeslennikov--Prop--Maximal Abelian subgroups} states that every maximal non-cyclic Abelian subgroup of $G_{k+1}$ is conjugate to some member of $\mathbb{P}_k \cup \{A\}$, so $G_{k+1}$ is indeed hyperbolic relative to its maximal non-cyclic Abelian subgroups.
\end{proof}

\subsection{Relative hyperbolic geometry}

Fix a relatively hyperbolic group $(G, \mathbb{P})$ with finite relative generating set $X$.
We call $\Cay(G, \xup)$ the \emph{relative Cayley graph}.

Recall that a metric space $(X, d_X)$ is \emph{$\delta$-hyperbolic}, or simply \emph{hyperbolic}, if it satisfies the \emph{thin triangles condition}: for any geodesic triangle with sides $\alpha, \beta, \gamma$, every point of $\alpha$ is $\delta$-close in the metric $d_X$ to some point of $\beta \cup \gamma$.

\begin{prop}[\cite{Osin2006}]
  Let $(G,\mathbb{P})$ be a relatively hyperbolic group.
  Then for any finite relative generating set $X$, the relative Cayley graph $\Cay(G,\xup)$ is hyperbolic.
\end{prop}

We have two distinct metrics on $\Cay(G,\xup)$.
The \emph{relative metric} is denoted $d_\xup$, and for $u,v \in \Cay(G, \xup)$, we define $d_\xup(u,v)$ to be the least number of edges in any path in $\Cay(G,\xup)$ having $u$ and $v$ as endpoints.
The \emph{absolute metric} is denoted $d_X$, and for $u,v \in \Cay(G,\xup)$, we define $d_X(u,v)$ to be the least number of edges in any $X$-labeled path in $\Cay(G,\xup)$ having $u$ and $v$ as endpoints.
Note that while $\Cay(G,\xup)$ is hyperbolic with respect to the relative metric, it will generally not be hyperbolic with respect to the absolute metric.

A \emph{relative geodesic} is an isometry $p: [0,L] \rightarrow (\Cay(G,\xup),d_\xup)$, where $[0,L]$ is a closed interval of real numbers.
We say that the \emph{endpoints} of $p$ are $p(0)$ and $p(L)$. Since every point $\Cay(G,\xup)$ is a distance at most 1 from some vertex, we will assume that $L$ is an integer and that $p$ maps integers to vertices.
For $u, v \in \Cay(G, \xup)$, we denote by $[u,v]_\xup$ a relative geodesic with endpoints $u$ and $v$.

Similarly, an \emph{absolute geodesic} is an isometry $p: [0,L] \rightarrow (\Cay(G,\xup),d_X)$.
We denote an absolute geodesic having $u$ and $v$ as endpoints by $[u,v]_X$.

A \emph{relative (absolute) broken geodesic} is a finite concatenation of relative (absolute) geodesics.
For a finite collection $\{a_1, \dots, a_k\}$ of points in $\Cay(G, \xup)$, we will denote by $[a_1, a_2, \dots, a_k]_\xup$ a broken relative geodesic which is the union of relative geodesics $\ds{\bigcup_{i=1}^{k-1} [a_i,a_{i+1}]_\xup}$.
Likewise, $[a_1, a_2, \dots, a_k]_X$ denotes the analogous broken absolute geodesic.

The \emph{length} of a path $\alpha$ in $\Cay(G,\xup)$, denoted $\len(\alpha)$, is the number of edges in the path.
Note that $\len([a,b]_\xup) = d_\xup(a,b)$ and $\len([a,b]_X) = d_X(a,b)$, for instance.

\begin{defn}[Fellow traveling] \index{fellow traveling}
  \label{Defn--Fellow traveling}
  Let $p, q: [0,L] \rightarrow (\Cay(G,\xup),d_\xup)$ be relative geodesics.
  We say that $p$ and $q$ are \emph{relative (absolute) $k$-fellow travelers} if $d_\xup(p(i),q(i)) \leq k$ (resp. $d_X(p(i),q(i)) \leq k$) for every integer $i$ in $[0,L]$.
  We say that $p$ and $q$ \emph{relatively (absolutely) $k$-fellow travel for a length of $L'$} if $p|_{[0,L']}$ and $q|_{[0,L']}$ are relative (absolute) $k$-fellow travelers.
\end{defn}

\begin{remark}
  Our notion of $k$-fellow traveling is often referred to in the literature as \emph{synchronous $k$-fellow traveling}, to distinguish it from \emph{asynchronouse $k$-fellow traveling}, which does not respect the parameterization of the geodesics.
  We will not require the notion of asynchronous $k$-fellow traveling here.
\end{remark}

\begin{defn}[Relatively quasiconvex] \index{relatively quasiconvex}
  \label{Defn--relatively quasiconvex}
  A subgroup $H$ of $(G, \mathbb{P})$ is called \emph{relatively quasiconvex} if there exists a constant $\epsilon > 0$ such that the following holds.
  Let $g, h \in H$ and let $[g,h]_\xup$ be an arbitrary relative geodesic in $\Cay(G, \xup)$.
  Then for every vertex $v \in [g,h]_\xup$, there exists a vertex $u \in H$ such that
  \begin{equation*}
    d_X(v,u) \leq \epsilon.
  \end{equation*}
\end{defn}

\begin{defn}[Strongly relatively quasiconvex] \index{strongly relatively quasiconvex}
  \label{Defn--strongly relatively quasiconvex}
  A relatively quasiconvex subgroup $H$ of $(G, \mathbb{P})$ is called \emph{strongly relatively quasiconvex} if the intersection $H \cap P^g$ is finite for any $g \in G$ and $P \in \mathbb{P}$.
\end{defn}

Osin notes in Proposition 4.10 of \cite{Osin2006} that the relative and strong relative quasiconvexity properties are invariant with respect to choice of finite generating set for $G$.

\begin{prop}[{\cite[4.19]{Osin2006}}]
  \label{Osin--Prop--Centralizers of hyperbolics are strongly relatively quasiconvex}
  Let $(G, \mathbb{P})$ be a relatively hyperbolic group, and let $u \in G$ be a hyperbolic element.
  Then the centralizer $C_G(u)$ is a strongly relatively quasiconvex subgroup of $G$.
\end{prop}

Let $\lambda > 0$ and $c \geq 0$. \index{quasi-isometric embedding}
Recall that a map of metric spaces $f: (X, d_X) \rightarrow (Y, d_Y)$ is a \emph{$(\lambda,c)$-quasi-isometric embedding} if for all $a, b \in X$, we have
\begin{equation*}
  \dfrac{1}{\lambda}d_X(a,b) - c \leq d_Y\big(f(a),f(b)\big) \leq \lambda d_X(a,b) + c.
\end{equation*}

\begin{prop}[\cite{Osin2006}]
  Every strongly relatively quasiconvex subgroup of $(G, \mathbb{P})$ is quasi-isometrically embedded in $\Cay(G,\mathbb{P})$.
\end{prop}

\begin{prop}[\cite{Osin2006}]
  Let $u$ be a hyperbolic element of $(G, \mathbb{P})$.
  Then $C_G(u)$ is cyclic.
\end{prop}

\begin{prop}[\cite{Osin2006}]
  \label{Osin--Prop--U is qie, lambda_u, c_u}
  For any hyperbolic $u \in (G, \mathbb{P})$ generating its own centralizer, there are constants $\lambda_u >0, c_u \geq 0$ such that
  \begin{equation}
    \dfrac{1}{\lambda_u}|n| - c_u \leq d_\xup(1,u^n) \leq \lambda_u |n| + c_u
  \end{equation}
  for all $n \in \mathbb{Z}$.
\end{prop}

\section{Main Results}

\subsection{Relative hyperbolic geometry}

We once again fix a relatively hyperbolic group $(G, \mathbb{P})$ with finite relative generating set $X$ such that the relative Cayley graph $\Cay(G,\xup)$ is $\delta$-hyperbolic.

\begin{lem}
  \label{Solie--Lemma--Double Coset FT, B_0}
  Let $u \in G$ be a hyperbolic element generating its own centralizer $U=C_G(u)$.
  There is a function $B_0: \mathbb{N} \rightarrow \mathbb{N}$ depending only on $(G, \mathbb{P})$, $X$, and $u$ such that the following holds.

  Let $g \in G-U$.
  Let $p,q \in U$ and $s,t \in gU$.
  For any $p',q' \in [p,q]_\xup$ and $s',t' \in [s,t]_\xup$ such that $[p',q']_\xup$ and $[s',t']_\xup$ are absolute $k$-fellow travelers, then
  \begin{equation*}
    d_\xup(p',q'),\: d_\xup(s',t') \leq B_0(k).
  \end{equation*}
\end{lem}

\begin{figure}
  \begin{center}
    \includegraphics{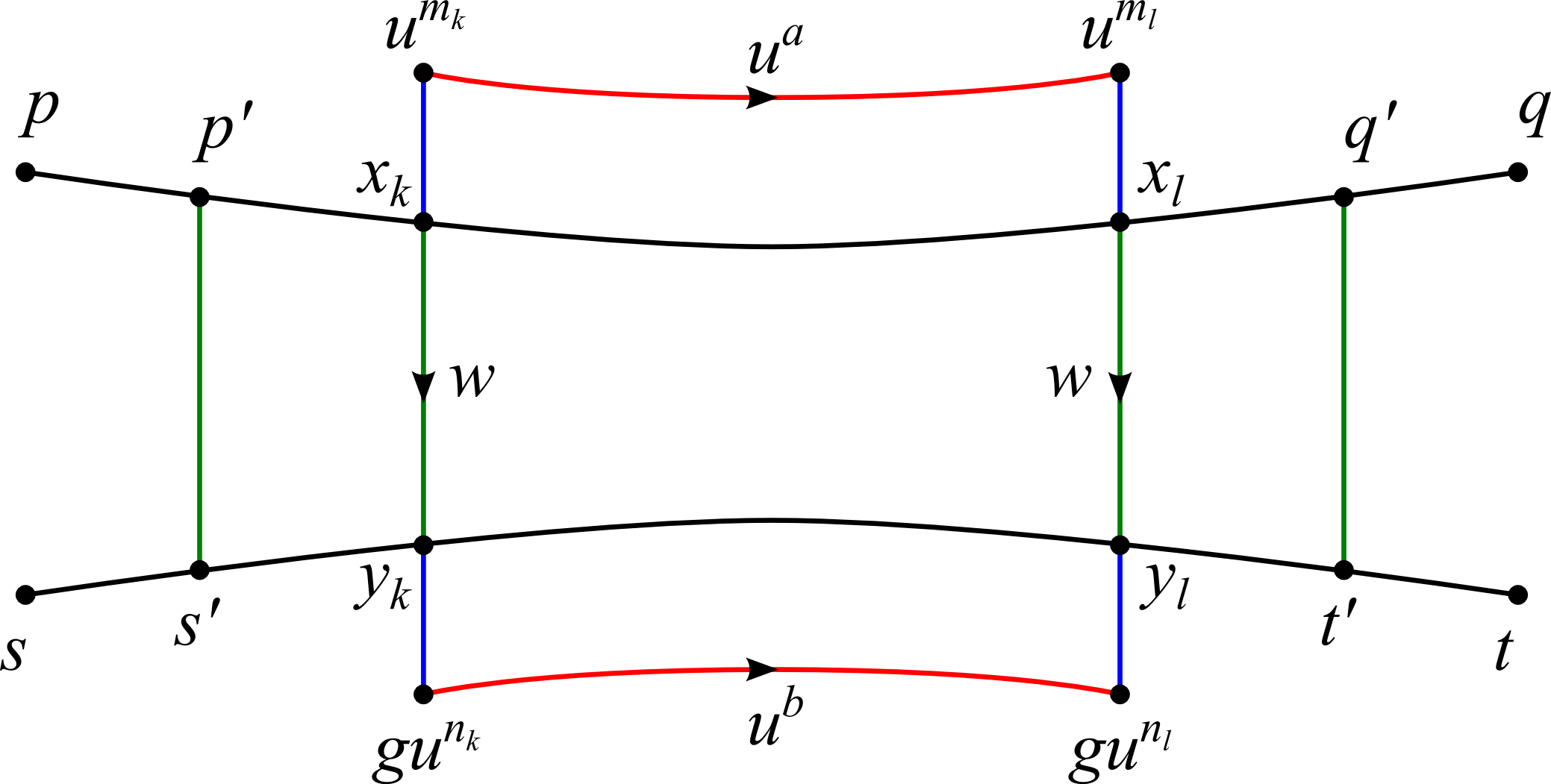}
    \caption{Producing the relation $w^{-1}u^a w = u^b$ in the proof of Lemma \ref{Solie--Lemma--Double Coset FT, B_0}}
    \label{Figure--DoubleCosetFT}
  \end{center}
\end{figure}

\begin{proof}
  Set $B_0(k) = (2\epsilon+1)(2|X|)^{k+2\epsilon}$, and suppose that for some nonnegative integer $k$, there exist $p,p',q,q',s,s',t,$ and $t'$ which satisfy the hypotheses but such that $d_\xup(p',q') > B_0(k)$.

  We may find $(2|X|)^{k+2\epsilon}$ vertices, denoted $x_i$, on $[p',q']_\xup$ such that if $i \neq j$ then $d_\xup(x_i,x_j) > 2\epsilon$.
  To each $x_i$ we may associate a $u^{m_i} \in U$ such that $d_X(x_i, y_i) \leq \epsilon$, since $U$ is relatively quasiconvex.
  Note that if $i \neq j$, then $m_i \neq m_j$; otherwise, we would have $d_\xup(x_i,x_j) \leq d_X(x_i,x_j) \leq 2\epsilon$, contradicting the choice of the $x_i$.

  Since $[p',q']_\xup$ and $[s',t']_\xup$ are absolute $k$-fellow travelers, for each $x_i$ there is a vertex $y_i \in [s',t']_\xup$ such that $d_X(x_i,y_i) \leq k$.
  Since $U$ is $\epsilon$-quasiconvex, for each $y_i$ there is $gu^{n_i} \in gU$ such that $d_X(y_i, gu^{n_i}) \leq \epsilon$.

  To each point $x_i$, we associate the broken absolute geodesic $[u^{m_i},x_i,y_i,gu^{n_i}]_X$.
  The length of such a path is at most $k+2\epsilon$, and there are $(2|X|)^{k+2\epsilon}$ such distinct paths, since no two of these paths have the same endpoint $u^{m_i}$.

  However, there are strictly fewer than $(2|X|)^{k+2\epsilon}$ distinct path labels for paths of length at most $k+2\epsilon$.
  Therefore, there are indices $k,l$ such that $[u^{m_k},x_k,y_k,gu^{n_k}]_X$ and $[u^{m_l},x_l,y_l,gu^{n_l}]_X$ have the same label, $w$.
  As the endpoints of these $w$-labeled paths differ by elements of $U$, we obtain a relation of the form $w^{-1} u^a w = u^b$ for some integers $a, b$.

  Since $G$ is relatively hyperbolic, we must have that $a = \pm b$ \cite[Corollary 4.21]{Osin2006}.
  Therefore, $w^2$ commutes with $u^a$.
  Since $G$ is a CSA-group and is therefore commutative-transitive (Proposition \ref{Myasnikov-Remeslennikov--Prop--CSA Properties}), $w$ commutes with $u$ and hence must be a power of $u$.
  This contradicts that $U$ and $gU$ are distinct cosets of $U$.
\end{proof}

\begin{lem}
  \label{Solie--Lemma--Coset Self FT, E_0}
  Let $u \in G$ be a hyperbolic element generating a maximal cyclic subgroup $U$.
  There is a function $E_0: \mathbb{N} \rightarrow \mathbb{N}$ depending only on $(G, \mathbb{P})$, $X$, and $u$ such that the following holds.

  For all $m, n \in \mathbb{Z}$ with $m < 0 < n$, the relative geodesics $[1,u^m]_\xup$ and $[1,u^n]_\xup$ relatively $k$-fellow travel for a length of at most $E_0(k)$.
\end{lem}

\begin{proof}
  If not, since $U$ is relatively quasiconvex and therefore quasi-isometrically embedded in $\Cay(G, \xup)$, there would have to be arbitrarily large powers of $u$ which have relative length bounded above by a constant.
  However, this contradicts that $U$ is quasi-isometrically embedded.
\end{proof}

Let $S$ be some set of elements of $(G, \mathbb{P})$.
We say that $g \in S$ is an \emph{$\xup$-shortest element of $S$} if $|g|_\xup \leq |h|_\xup$ for every $h \in S$.

\begin{lem}
  \label{Solie--Lemma--Coset FT, C_0}
  Let $u \in G$ generate a cyclic hyperbolic subgroup $U$.
  There is a function $C_0: \mathbb{N} \rightarrow \mathbb{N}$ depending only on $(G, \mathbb{P})$, $X$, and $u$ such that the following holds.

  Let $h$ be an $\xup$-shortest element of $hU$.
  Then for any integer $n$, the geodesics $[h,1]_\xup$ and $[h,hu^n]_\xup$ absolutely $k$-fellow travel for no longer than $C_0(k)$.
\end{lem}

\begin{figure}
  \begin{center}
    \includegraphics{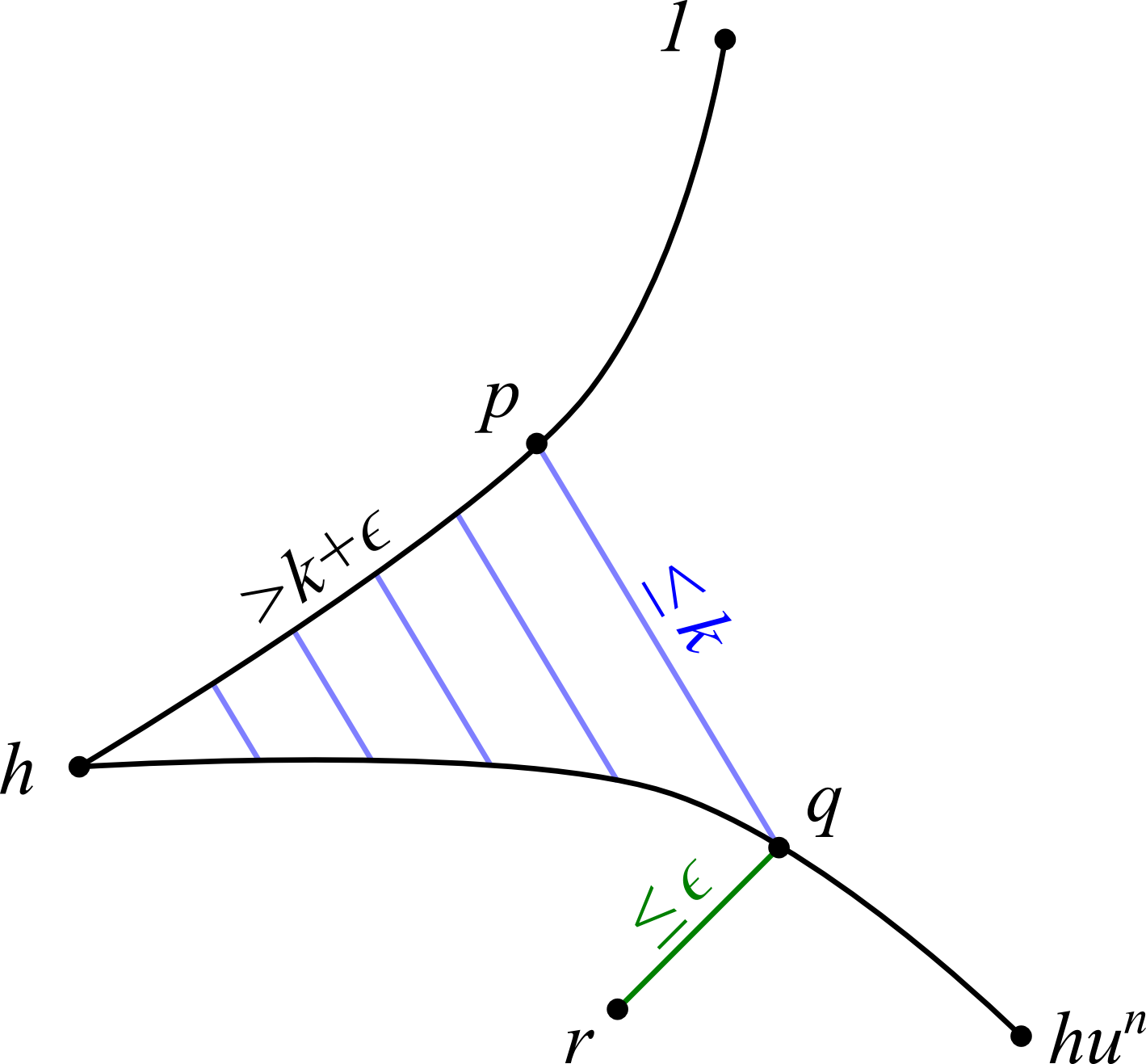}
  \end{center}
  \caption{Finding a shorter coset representative in Lemma \ref{Solie--Lemma--Coset FT, C_0}}
  \label{Figure--C_0 Lemma}
\end{figure}

\begin{proof}
  Suppose that for fixed $k$ and $n$, $[h,1]_\xup$ and $[h,hu^n]_\xup$ absolutely $k$-fellow travel for longer than $k+\epsilon$.
  Then there is a vertex $p \in [h,1]_\xup$ with $d_\xup(h,p) > k+\epsilon$ and such that there exists $w \in [h, hu^n]_\xup$ with $d_X(p,q) \leq k$.
  Since $U$ is relatively quasiconvex with constant $\epsilon$, there is a vertex $r \in hU$ with $d_X(q,r) \leq \epsilon$.
  Then $[1,p,q,r]_\xup$ is a broken relative geodesic of length at most $d_\xup(1,p)+k+\epsilon < d_\xup(1,h)$, contradicting that $h$ is amongst the $d_\xup$-shortest elements of $hU$. (See Figure \ref{Figure--C_0 Lemma}.)
\end{proof}

\begin{remark}
  The analogous statement holds for elements $h$ which are $\xup$-shortest in the coset $Uh$.
  Moreover, also note that if $h$ is $\xup$-shortest in $UhU$, then $h$ is $\xup$-shortest in both $Uh$ and $hU$.
\end{remark}

\begin{prop}[\cite{Osin2005}]
  \label{Osin--Prop--Relative geodesic triangles have absolute resolutions}
  Let $(G, \mathbb{P})$ be relatively hyperbolic with finite relative generating set $X$.
  There exist constants $\rho, \sigma > 0$ having the following property.

  Let $\Delta$ be a triangle with vertices $x, y, z$ whose sides $[x,y]_\xup, [y,z]_\xup,[x,z]_\xup$ are relative geodesics in $\Cay(G, \xup)$.
  Suppose that $u$ and $v$ are vertices on $[x,y]_\xup$ and $[x,z]_\xup$ respectively such that
  \begin{equation*}
    d_\xup(x,u) = d_\xup(x,v)
  \end{equation*}
  and
  \begin{equation*}
    d_\xup(u,y)+d_\xup(v,z) \geq d_\xup(y,z)+\sigma.
  \end{equation*}
  Then
  \begin{displaymath}
    d_X(u,v) \leq \rho.
  \end{displaymath}
\end{prop}

\begin{figure}
  \begin{center}
    \includegraphics{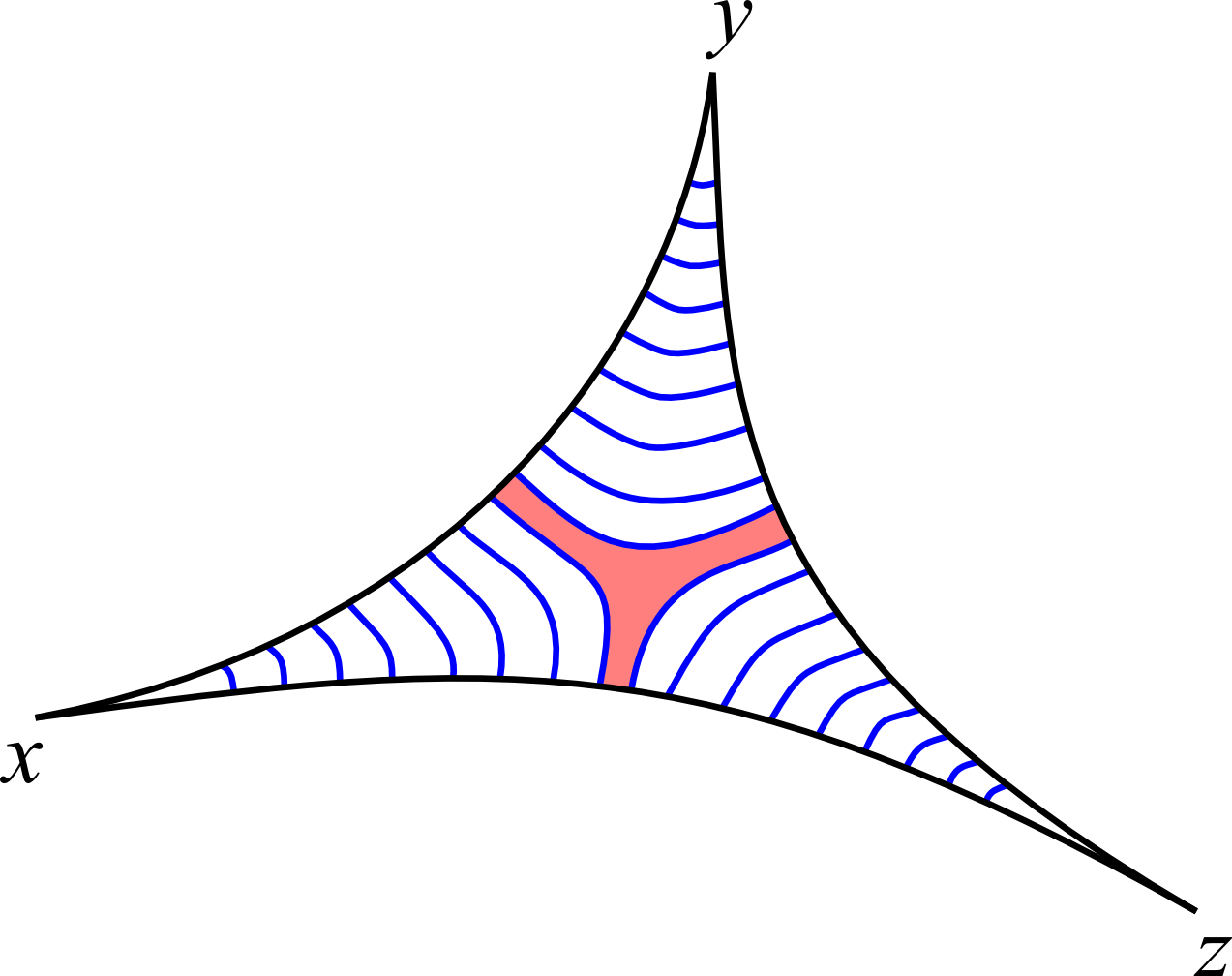}
  \end{center}
  \caption[A relative geodesic triangle.]{A relative geodesic triangle. The shaded lines join pairs of points on the triangle which are $\rho$-close in the absolute metric. The shaded area represents the region where the absolute $\rho$-fellow traveling property may fail.}
\end{figure}

Recall that if $x, y,$ and $z$ are vertices in $\Cay(G, \xup)$, then the \emph{Gromov inner product} is defined as
\begin{equation*}
  \langle y|z \rangle_x := \dfrac{1}{2}(d_\xup(x,y) + d_\xup(x,z) - d_\xup(y,z)).
\end{equation*}

\begin{cor}
  \label{Solie--Corollary--Relative geodesic triangles absolutely FT}
  Let $\rho, \sigma, x, y, z$ be as in Proposition \ref{Osin--Prop--Relative geodesic triangles have absolute resolutions}.
  Then adjacent sides $[x,y]_\xup$ and $[x,z]_\xup$ absolutely $\rho$-fellow travel for length at least $\langle y|z \rangle_x - \sigma/2$.
\end{cor}

\begin{proof}
  Let $u \in [x,y]_\xup$ and $v \in [x,z]_\xup$ be such that $d_\xup(x,u)=d_\xup(x,v)=\ell$ and $d_\xup(u,y)+d_\xup(v,z) \geq d_\xup(y,z)+\sigma.$
  We then have
  \begin{displaymath}
    d_\xup(u,y)+d_\xup(v,z) = d_\xup(x,y)+d_\xup(x,z)-2\ell.
  \end{displaymath}
  Further,
  \begin{align*}
    d_\xup(x,y)+d_\xup(x,z) - 2\ell & \geq d_\xup(y,z) + \sigma \\
    d_\xup(x,y)+d_\xup(x,z)-d_\xup(y,z) - 2\ell & \geq \sigma\\
    2 \langle y|z \rangle_x - 2\ell & \geq \sigma \\
    \langle y|z \rangle_x - \sigma/2 & \geq \ell.
  \end{align*}

  Therefore, if $\ell \leq \langle y|z \rangle_x - \sigma/2$, then $u$ and $v$ satisfy the hypotheses of Proposition \ref{Osin--Prop--Relative geodesic triangles have absolute resolutions} and are therefore $\rho$-close in the absolute metric.
\end{proof}

For a given relative geodesic triangle with vertices $x, y, z$, the \emph{center} of the side $[x,y]_\xup$ is the point $c \in [x,y]_\xup$ such that $d_\xup(x,c) = \langle y|z \rangle_x$ and $d_\xup(y,x) = \langle x|z \rangle_y$.

\begin{lem}
  \label{Solie--Lemma--U Lemma, F_0}
  Let $u \in G$ generate a maximal cyclic hyperbolic subgroup $U$.
  Let $g \in G$, and let $h \in G$ be a $\xup$-shortest element of $UgU$.
  There is a constant $F_0$ depending only on $(G,\mathbb{P})$, $X$, and $u$ such that the following holds.

  Suppose that we have $m$ and $n$ such that $g = u^m h u^n$.
  Then $[1,u^m]_\xup$ and $[u^mh,u^mhu^n]_\xup$ each absolutely $2\rho$-fellow travel $[1,u^mhu^n]_\xup$ from their respective shared endpoints for all but at most $F_0$ of their length.
\end{lem}

\begin{proof}
  Let $Q$ be the relative geodesic quadrilateral with sides $[1,u^m]_\xup$,  ${[u^m, u^mh]_\xup}$, $[u^mh,u^mhu^n]_\xup$, and $[1,u^mhu^n]_\xup$.


  By drawing a relative geodesic diagonal for $Q$, we obtain two relative geodesic triangles.
  As in Proposition \ref{Osin--Prop--Relative geodesic triangles have absolute resolutions}, every pair of sides in either of these triangles absolutely $\rho$-fellow travel from their common vertex for a length of at least their Gromov inner product minus $\sigma/2$.

  \begin{figure}
    \begin{center}
      \includegraphics{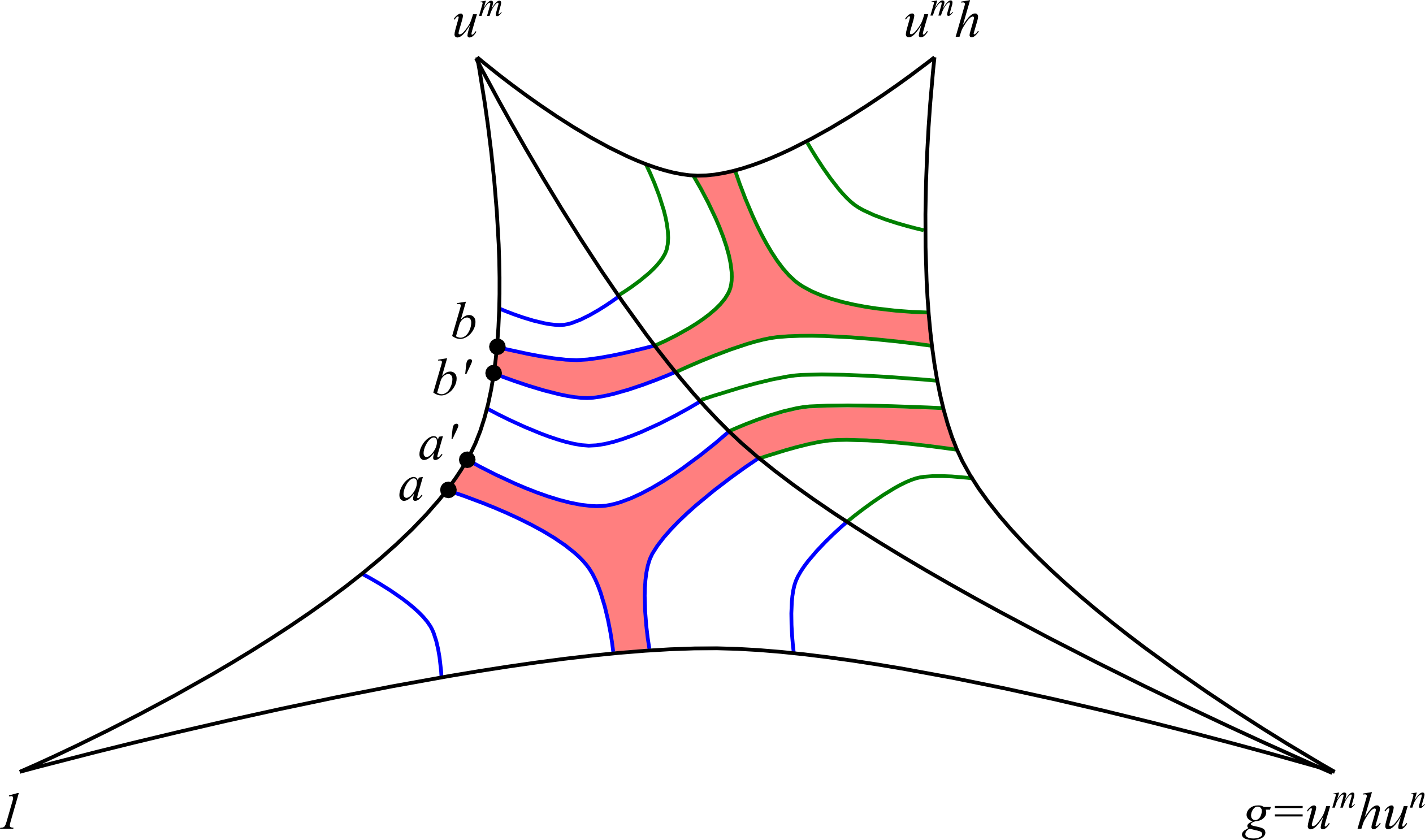}
    \end{center}
    \caption{A decomposition of $Q$ and one of its sides.}
    \label{Fig--Quadrilateral decomposition}
  \end{figure}

  We extend the fellow-traveling property of the sides of these triangles to the sides of $Q$. (See Figure \ref{Fig--Quadrilateral decomposition} for one configuration of such an extension; the shaded area represents the area near the centers of the triangles where absolute fellow traveling is not guaranteed.)
  We see that there exist vertices $a,a',b,b' \in [1,u^m]_\xup$ such that:
  \begin{enumerate}
    \item
      The subpath $[1,a]_\xup$ and some initial subpath of $[1,u^mhu^n]_\xup$ absolutely $2\rho$-fellow travel;
    \item
      The subpath $[u^m,b]_\xup$ and some initial subpath of $[u^m,u^mh]_\xup$ absolutely $2\rho$-fellow travel;
    \item
      The subpath $[a',b']_\xup$ absolutely $2\rho$-fellow travels some subpath of $[u^mhu^n,u^mh]_\xup$; and
    \item
      The relative lengths of the subpaths $[a,a']_\xup$ and $[b',b]_\xup$ do not exceed $\sigma$.
  \end{enumerate}

  We are interested in the total length of the subpath $[a,u^m]_\xup$, since, as noted, $[1,a]_\xup$ fellow travels with a subpath of $[1,u^mhu^n]_\xup$.
  Observation (2) above implies that the length of $[u^m,b]_\xup$ is at most $C_0(2\rho)$, by Lemma \ref{Solie--Lemma--Coset FT, C_0}.
  Observation (3) implies that the length of $[a',b']_\xup$ is at most $B_0(2\rho)$, by Lemma \ref{Solie--Lemma--Double Coset FT, B_0}.

  Consequently, we have that
  \begin{displaymath}
    \len([a,u^m]_\xup) \leq B_0(2\rho) + C_0(2\rho) + 2\sigma =: F_0.
  \end{displaymath}
\end{proof}

\begin{lem}
  \label{Solie--Cor--Passing to minimal double coset representatives}
  Let $u, g, h, m,$ and $n$ be as in Lemma \ref{Solie--Lemma--U Lemma, F_0}.
  Then we have
  \begin{displaymath}
    \len([1,u^m,u^mh,u^mhu^n]_\xup) \leq 3|g|_\xup + 2F_0.
  \end{displaymath}
\end{lem}

\begin{proof}
  The lengths of the subpaths $[1,u^m]_\xup$ and $[u^mh,u^mhu^n]_\xup$ are bounded above by $|g|_\xup+F_0$ by Lemma \ref{Solie--Lemma--U Lemma, F_0}.
  Since $h$ is a $\xup$-shortest representative of $UgU$, we have $|h|_\xup \leq |g|_\xup$, and so the length of $[u^m,u^mh]_\xup$ is at most $|g|_\xup$.
\end{proof}

Let $\mathbf{r} = (r_0, r_1, \dots, r_k)$ be a tuple of integers.
We define
\begin{equation*}
  \min(\mathbf{r}) := \min_i |r_i|.
\end{equation*}

\begin{lem}
  \label{Solie--Lemma--Padding for short multiwords}
  Let $(G, \mathbb{P})$ be a relatively hyperbolic group with finite generating set $X$, and let $U$ be a subgroup generated by a hyperbolic element $u \in G$.
  There exists a positive integer $N_0$ depending only on $(G, \mathbb{P})$, $X$, and $u$ such that the following holds.

  Let $\mathbf{h} = (h_1, h_2, \dots, h_k)$ be a tuple of elements of $X$ such that each $h_i$ is $\xup$-shortest in the double coset $U h_i U \neq U$, and let $\mathbf{r} = (r_0, r_1, \dots, r_k)$ be a tuple of integers.
  Define
  \begin{equation*}
    w_{\mathbf{h}}(\mathbf{r}) := u^{r_0} h_1 u^{r_1} h_2 u^{r_2} \cdots u^{r_{k-1}} h_k u^{r_k}.
  \end{equation*}
  Then $w_{\mathbf{h}}(\mathbf{r}) \neq 1$ in $G$ for all $\mathbf{r}$ such that $\min(\mathbf{r}) > N_0$.
\end{lem}

\begin{proof}
  Let $\alpha$ be a path in $\Cay(G,\xup)$ labeled by
  \begin{displaymath}
    (u^{r_0} * h_1 * u^{\left\lfloor r_1/2 \right\rfloor}) * (u^{\left\lceil r_1/2 \right\rceil} * h_2 * u^{\left\lfloor r_2/2 \right\rfloor}) * \cdots * (u^{\left\lceil r_{k-1}/2 \right\rceil} * h_k * u^{r_k}),
  \end{displaymath}
  where $*$ denotes concatenation of words (as opposed to concatenation followed by free reduction) and $\lfloor \cdot \rfloor, \lceil \cdot \rceil$ are the usual floor and ceiling functions.
  Let $\alpha_1$ be the subpath labeled by $u^{r_0}*h_1*u^{\left\lfloor r_1/2 \right\rfloor}$ and $\alpha_k$ the subpath labeled by $u^{\left\lceil r_{k-1}/2 \right\rceil} * h_k * u^{r_k}$, and for each $i=2, \dots, k-1$, let $\alpha_i$ be the subpath of $\alpha$ labeled by $u^{\left\lceil r_{i-1}/2 \right\rceil}*h_i*u^{\left\lfloor r_i/2 \right\rfloor}$.
  The path $\alpha$ is then the concatenation of the $\alpha_i$.
  Further define the vertices $v_{i-1}$ and $v_i$ to be the endpoints of $\alpha_i$ for each $i$.
  Finally, for each $i$, define $\beta_i$ to be a relative geodesic $[v_{i-1}, v_i]_\xup$, and define $\beta$ to be the broken relative geodesic which is the concatenation of the $\beta_i$.
  (See Figure \ref{Fig--Alpha decomposition}.)

  \begin{figure}
    \begin{center}
      \includegraphics{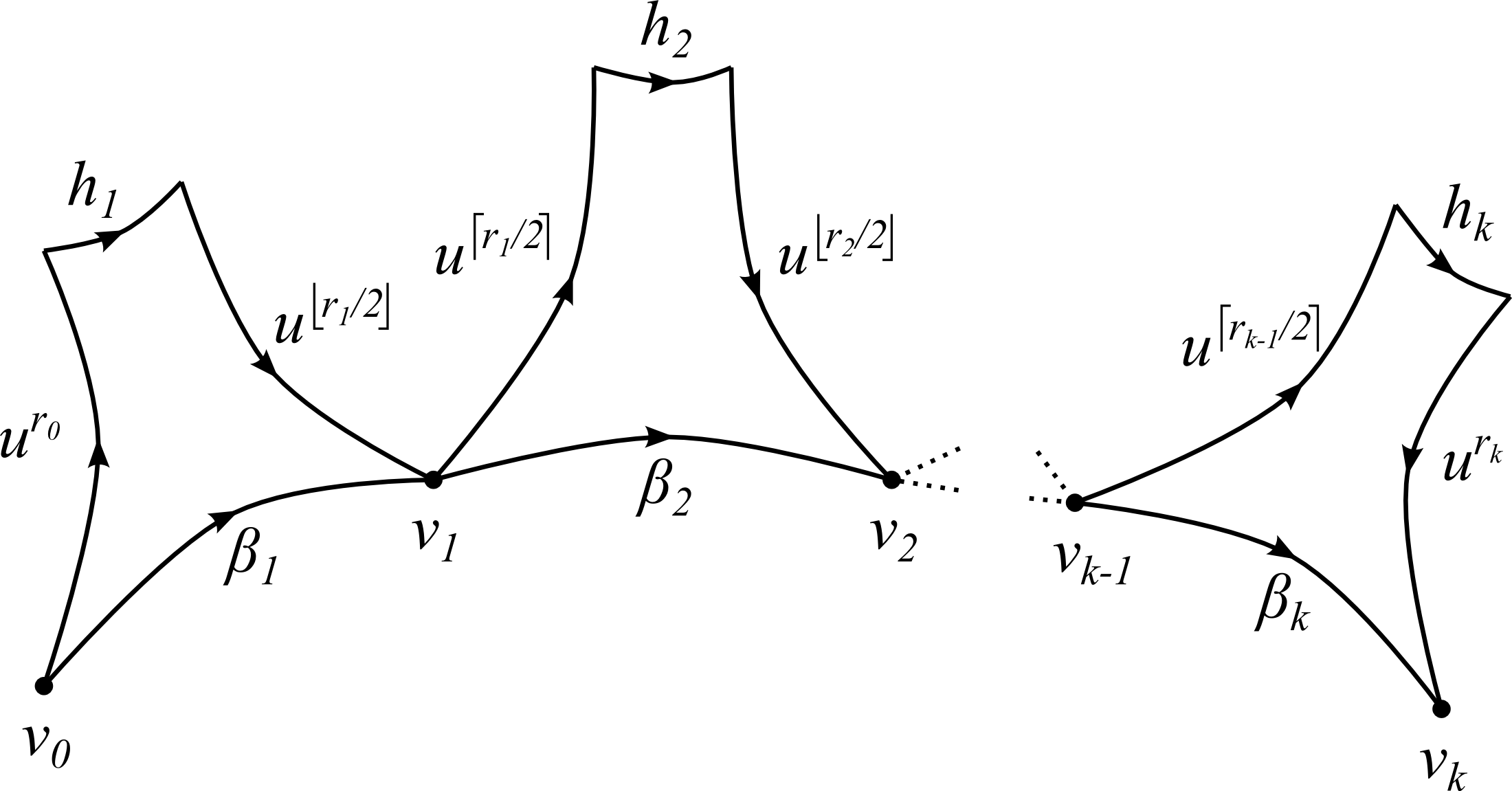}
    \end{center}
    \caption{The decomposition of $\alpha$.}
    \label{Fig--Alpha decomposition}
  \end{figure}

  \begin{lem}
    \label{Solie--Lemma--Length of beta_i}
    For each $i$ and $n$ we have
    \begin{equation}
      \dfrac{2}{\lambda_u} \lfloor \min(\mathbf{r})/2 \rfloor -2c_u - 2F_0 \leq \len(\beta_i).
    \end{equation}
  \end{lem}


  \begin{proof}
    This follows directly from Proposition \ref{Osin--Prop--U is qie, lambda_u, c_u} and Lemma \ref{Solie--Lemma--U Lemma, F_0}.
  \end{proof}

  \begin{prop}
    For all $\mathbf{r}$ with
    \begin{equation}
      \label{Eqn--Constraint on beta fellow traveling}
      \lfloor \min(\mathbf{r})/2 \rfloor > \lambda_u(E_0(4\rho+\delta) + F_0 + c_u)
    \end{equation}
    and $1 \leq i < k$, $\beta_i$ and $\beta_{i+1}$ relatively $\delta$-fellow travel for a length of at most $E_0(4\rho+\delta)$ from their common endpoint $v_i$.
  \end{prop}

  \begin{proof}
    Suppose there is an $\mathbf{r}$ satisfying \eqref{Eqn--Constraint on beta fellow traveling} and $i$ such that $\beta_i$ and $\beta_{i+1}$ relatively $\delta$-fellow travel for a length longer than $E_0(4\rho+\delta)$.
    By construction, there are relative geodesics $\gamma_{i-1}$ and $\gamma_i$ starting at $v_i$ labeled by $u^{-\lfloor r_i/2 \rfloor}$ and $u^{\lceil r_i/2 \rceil}$ respectively.
    These relative geodesics absolutely $2\rho$-fellow travel $\beta_i$ and $\beta_{i+1}$ for all but at most $F_0$ of their length.
    By choice of $\mathbf{r}$ and Corollary \ref{Osin--Prop--U is qie, lambda_u, c_u}, $\gamma_j$ and $\beta_j$ are absolute $2\rho$-fellow travelers for a length of at least $E_0(4\rho+\delta)$ for $j=i,i+1$.

    \begin{figure}
      \begin{center}
        \includegraphics{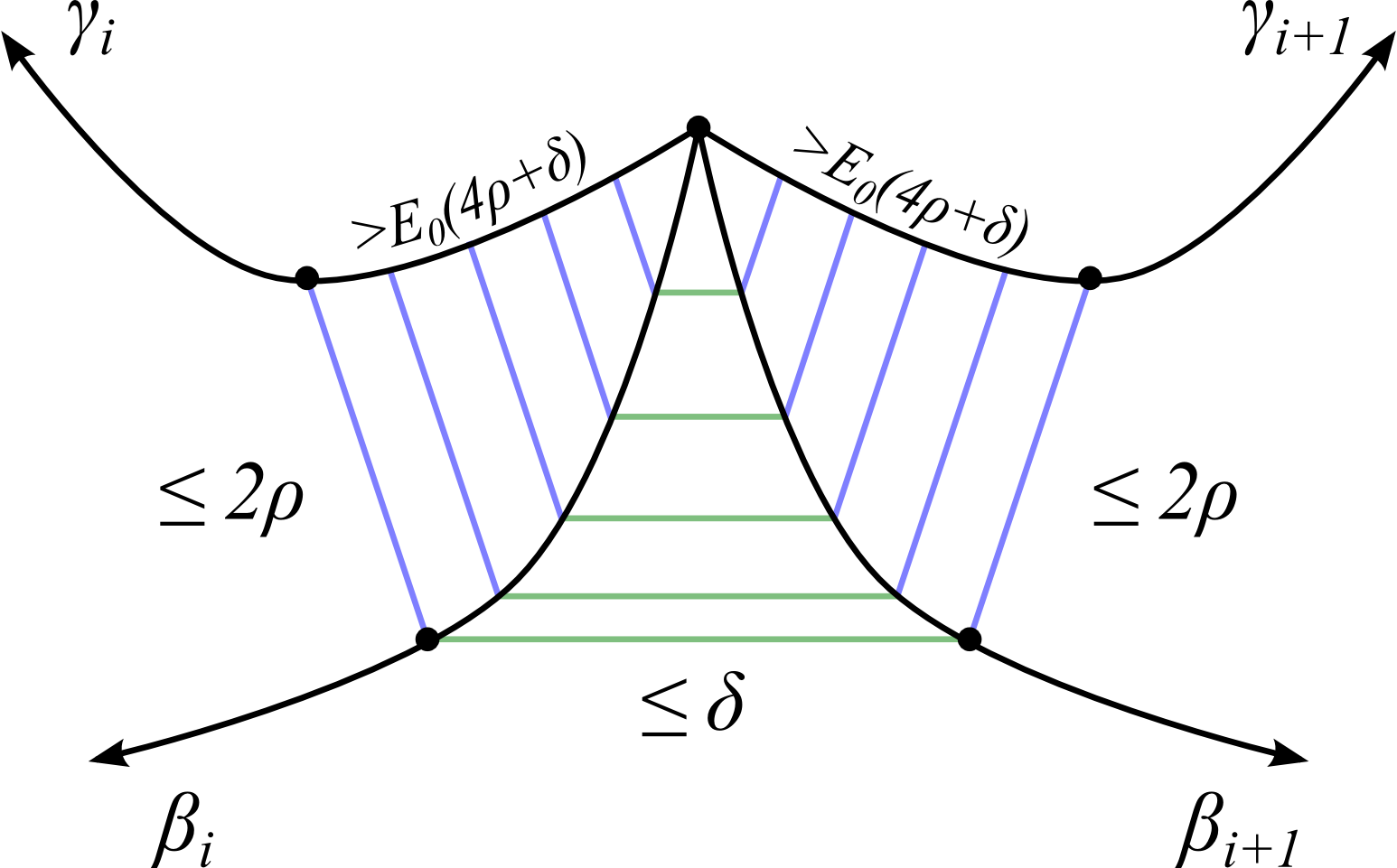}
      \end{center}
      \caption{$\beta_i$ and $\beta_{i+1}$ cannot fellow travel too far without causing $\gamma_i$ and $\gamma_{i+1}$ to fellow travel.}
      \label{Fig--Angle Lemma}
    \end{figure}

    However, if $\beta_i$ and $\beta_{i+1}$ are relative $\delta$-fellow travelers for longer than $E_0(4\rho+\delta)$, then $\gamma_i$ and $\gamma_{i+1}$ are relative $(4\rho+\delta)$-fellow travelers for longer than $E_0(4\rho+\delta)$, contradicting Lemma \ref{Solie--Lemma--Coset Self FT, E_0}. (See Figure \ref{Fig--Angle Lemma}.)
  \end{proof}

  Note that in a relative geodesic triangle, adjacent sides relatively $\delta$-fellow travel for a length of at least the Gromov inner product.
  This fellow traveling property allows us to show that the concatenation of relative geodesic segments is a quasi-geodesic with parameters depending on the Gromov inner product.

  \begin{prop}
    \label{Solie--Prop--Adjacent geodesics are quasigeodesics}
    Let $x,y,z \in \Cay(G,\xup)$.
    Then every subpath of the broken relative geodesic $[x,y,z]_\xup$ is a $(1, 2\langle x|z \rangle_y+2\delta)$-quasigeodesic.
  \end{prop}

  Proposition \ref{Solie--Prop--Adjacent geodesics are quasigeodesics} shows that for $\mathbf{r}$ satisfying \eqref{Eqn--Constraint on beta fellow traveling}, every adjacent pair of relative geodesics $\beta_i$ and $\beta_{i+1}$ is a relative $(1,2E_0(4\rho+\delta)+2\delta)$-quasigeodesic.

  \begin{prop}[{\cite[Lemma 4.8]{Neumann1995}}]
    \label{Neumann-Shapiro--Prop--Local quasigeodesics are quasigeodesics}
    Let $Y$ be a $\delta$-hyperbolic space.
    Given quasigeodesity constants $(\lambda, c)$, there exist $\kappa, \lambda',$ and $c'$ such that every $k$-local $(\lambda, c)$-quasigeodesic is a $(\lambda', c')$-quasigeodesic.
  \end{prop}

  \begin{prop}
    Let $\kappa,\lambda',c'$ be such that in $\Cay(G, \xup)$, every $\kappa$-local $(1,2E_0(4\rho+\delta)+2\delta)$-quasigeodesic is a $(\lambda',c')$-quasigeodesic.
    Let $\mathbf{r}$ satisfy \eqref{Eqn--Constraint on beta fellow traveling} and further assume that
    \begin{equation}
      \label{Eqn--beta longer than kappa}
      \lfloor \min(\mathbf{r})/2 \rfloor \geq \lambda_u \left(\dfrac{\kappa}{2} + c_u + F_0 \right).
    \end{equation}
    Then $\beta$ is a $(\lambda',c')$-quasigeodesic.
  \end{prop}

  \begin{proof}
    By Proposition \ref{Solie--Prop--Adjacent geodesics are quasigeodesics}, for every $i$, the broken geodesic $\beta_i \cup \beta_{i+1}$ is a $(1,2E_0(4\rho+\delta)+2\delta)$-quasigeodesic.
    The inequality \eqref{Eqn--beta longer than kappa} implies that the length of each $\beta_i$ is larger than $\kappa$.
    Every subpath of $\beta$ of length at most $\kappa$ is contained in $\beta_i \cup \beta_{i+1}$ for some $i$, and is therefore a relative $(1,2E_0(4\rho+\delta)+2\delta)$-quasigeodesic.
    The conclusion then follows from applying Proposition \label{Neumann-Shapiro--Prop--Local quasigeodesics are quasigeodesics}.
  \end{proof}

  Now let $\mathbf{r}$ be such that
  \begin{equation}
    \label{Eqn--Constraint quasigeodesic}
    \lfloor \min(\mathbf{r})/2 \rfloor > \lambda_u\left(\dfrac{c'}{2}+c_u+F_0\right).
  \end{equation}
  Then the length of each $\beta_i$ is at least $c'$, and so the length of $\beta$ is at least $c'$.
  The broken relative geodesic $\beta$, which is also a $(\lambda', c')$-quasigeodesic, therefore has necessarily distinct endpoints.
  Since $\alpha$ has the same endpoints as $\beta$ and is labeled by $w_{\mathbf{h}}(\mathbf{r})$, we have $w_{\mathbf{h}}(\mathbf{r}) \neq 1$ in $G$.

  Let $N_{-1}$ be an integer larger than the right hand side in the inequalities \eqref{Eqn--Constraint on beta fellow traveling}, \eqref{Eqn--beta longer than kappa}, and \eqref{Eqn--Constraint quasigeodesic}.
  Pick an integer $N_0$ such that $N_0 > 2N_{-1}+2$.
  Then for all $\mathbf(r)$ with $\min(\mathbf{r}) > N_0$, we have that $\lfloor \min(\mathbf{r})/2 \rfloor > N_{-1}$.
  Thus $N_0$ is the promised constant.
\end{proof}

\begin{lem}
  \label{Solie--Lemma--Padding for multiwords}
  Let $(G, \mathbb{P})$ be a relatively hyperbolic group with finite generating set $X$, and let $U$ be a subgroup generated by a hyperbolic element $u \in G$.
  There is a linear function $N_1: \mathbb{N} \rightarrow \mathbb{N}$ such that the following holds.

  Let $\mathbf{g}=(g_1, g_2, \dots, g_k)$ be a tuple of $X$-words such that $\ds{\sum_{i=1}^k |g_i|_X} \leq R$ and $g_i \in G-U$ for all $i$.
  For any tuple of integers $\mathbf{r} = (r_0, \dots, r_k)$, define
  \begin{equation}
    \label{Eqn--Padded general multiword}
    w_{\mathbf{g}}(\mathbf{r}) := u^{r_0} g_1 u^{r_1} g_2 u^{r_2} \cdots u^{r_{k-1}} g_k u^{r_k}.
  \end{equation}
  Then we have $w_{\mathbf{g}}(\mathbf{r}) \neq 1$ in $G$ for all $\mathbf{r}$ such that $\min(\mathbf{r}) > N_1(R)$.
\end{lem}

\begin{proof}
  Consider a single $g_i$.
  We may write $g_i = u^{s_i} h_i u^{t_i}$ with $h_i$ a $\xup$-shortest element of $Ug_iU$.
  By Lemma \ref{Solie--Lemma--U Lemma, F_0}, we have
  \begin{equation*}
    |u^{s_i}|_\xup, |u^{t_i}|_\xup \leq |g_i|_X + F_0 \leq R + F_0.
  \end{equation*}
  Using the constants $\lambda_u$ and $c_u$ from Proposition \ref{Osin--Prop--U is qie, lambda_u, c_u}, define
  \begin{equation*}
    N_1(R) := N_0 + 2\lambda_u(R+F_0+c_u),
  \end{equation*}
  where $N_0$ is the constant from Theorem \ref{Solie--Lemma--Padding for short multiwords}.
  Note that $\lambda_u(R+F_0+c_u) > |s_i|, |t_i|$ for all $i$.

  Let $\mathbf{r} = (r_0, r_1, \dots, r_k)$ be a tuple of integers with $\min(\mathbf{r}) > N_1(R)$.
  Then we have
  \begin{align}
    \label{Eqn--word with padding}
    w_{\mathbf{g}}(\mathbf{r}) &= u^{r_0} g_1 u^{r_1} g_2 u^{r_2} \cdots u^{r_{k-1}} g_k u^{r_k} \notag \\
      &= u^{r_0} (u^{s_1} h_1 u^{t_1}) u^{r_1} (u^{s_2} h_2 u^{t_2}) u^{r_2} \cdots u^{r_{k-1}} (u^{s_k} h_k u^{t_k}) u^{r_k}) \notag \\
      &= (u^{r_0 + s_1}) h_1 (u^{t_1+r_1+s_2}) h_2 (u^{t_2+r_2+s_3}) \cdots (u^{t_{k-1}+r_{k-1}+s_k}) h_k (u^{t_k+r_k})
  \end{align}
  where every exponent of $u$ appearing in \eqref{Eqn--word with padding} has magnitude at least $N_0$.
  By Theorem \ref{Solie--Lemma--Padding for short multiwords}, $w_{\mathbf{g}}(\mathbf{r})$ is nontrivial in $G$.
\end{proof}

\begin{lem}
  \label{Solie--Lemma--Padding for more general words}
  Let $(G, \mathbb{P})$ be a relatively hyperbolic group with finite generating set $X$, and let $U$ be a subgroup generated by a hyperbolic element $u \in G$.
  There is a linear function $N_2: \mathbb{N} \rightarrow \mathbb{N}$ such that the following holds.

  Let $\mathbf{g}=(g_1, g_2, \dots, g_k)$ be a tuple of $X$-words, and let $g_0, g_{k+1}$ be $X$-words such that $\ds{\sum_{i=0}^{k+1} |g_i|_X} \leq R$ and $g_i \in G-U$ for all $i$.
  Let $\mathbf{r} = (r_0, \dots, r_k)$ be a tuple of integers and define
  \begin{equation*}
    w_{\mathbf{g}}(\mathbf{r}) = u^{r_0} g_1 u^{r_1} g_2 u^{r_2} \cdots u^{r_{k-1}} g_k u^{r_k}.
  \end{equation*}
  Then for all $\mathbf{r}$ such that $\min(\mathbf{r}) > N_2(R)$, the elements
    \begin{align*}
      &w_{\mathbf{g}}(\mathbf{r}),\\
      g_0 &w_{\mathbf{g}}(\mathbf{r}),\\
      &w_{\mathbf{g}}(\mathbf{r}) g_{k+1}, \text{ and }\\
      g_0 &w_{\mathbf{g}}(\mathbf{r}) g_{k+1}
    \end{align*}
  are all nontrivial in $G$.
\end{lem}

\begin{proof}
  Note that if
  \begin{equation*}
    \min(\mathbf{r}) > 2\lambda_u\left(2\lambda' R+c_u+F_0+c'\right)+2,
  \end{equation*}
  then
  \begin{equation*}
    \lfloor \min(\mathbf{r})/2 \rfloor > \lambda_u \left(2\lambda'R+c_u+F_0+c'\right)
  \end{equation*}
  and therefore $|w_{\mathbf{g}}(\mathbf{r})|_\xup > 2R$.
  Define
  \begin{equation*}
    N_2(R) := N_1(R) + 2\lambda_u\left(2\lambda' R+c_u+F_0+c'\right)+2,
  \end{equation*}
  and note that since $N_1$ is linear in $R$, so is $N_2$.

  Then for all $\mathbf{r}$ with $\min(\mathbf{r}) > N_2(R)$, we have $|w_{\mathbf{g}}(\mathbf{r})|_\xup > 2R \geq |g_0|_\xup + |g_{k+1}|_\xup$, and so none of the promised words are trivial in $G$ by Lemma \ref{Solie--Lemma--Padding for multiwords}.
\end{proof} 

\subsection{Discriminating complexity}

Let $H$ be a finitely generated group, and let $G$ be a fully residually $H$ group.
Let $X$ and $Y$ be fixed finite generating sets for $G$ and $H$, respectively.

\begin{defn}[Complexity] \index{complexity of a homomorphism}
  \label{Defn--Complexity}
  Let $\phi: G \rightarrow H$.
  The \emph{complexity} of $\phi$ with respect to the finite generating sets $X$ and $Y$ is
  \begin{displaymath}
    |\phi|_X^Y := \displaystyle{\max_{x \in X} |\phi(x)|_Y}.
  \end{displaymath}
\end{defn}

The following lemma is straightforward to verify.

\begin{lem}
  \label{Solie--Lemma--Change of generating set}
  Let $\phi: G \rightarrow H$ and $\theta: H \rightarrow K$ and let $X$, $Y$, and $Z$ be finite generating sets for $G, H,$ and $K$, respectively.
  Then
  \begin{equation*}
    |\theta \circ \phi|_X^Z \leq |\phi|_X^Y \cdot |\theta|_Y^Z.
  \end{equation*}
\end{lem}

\begin{remark}
  Using the above convention, if $X'$ and $Y'$ are alternate finite generating sets for $G$ and $H$, respectively, we have
  \begin{equation*}
    |\phi|_{X'}^{Y'} \leq |\id|_{X'}^X \cdot |\phi|_X^Y \cdot |\id|_Y^{Y'}.
  \end{equation*}
\end{remark}

Since $G$ is fully residually $H$, for every $R \in \mathbb{N}$, there is a homomorphism $\phi_R$ which $H$-discriminates the finite set $B_R(G,X)-1$.

\begin{defn}[Discriminating complexity] \index{discriminating complexity}
  \label{Defn--Discriminating complexity}
  Define a function $C_{G,X}^{H,Y}: \mathbb{N} \rightarrow \mathbb{N}$ via
  \begin{displaymath}
    C_{G,X}^{H,Y}(R) := \min \{|\phi|_X^Y \;:\; (\phi: G \rightarrow H) \text{ discriminates } (B_R(G,X)-1) \}.
  \end{displaymath}
  The function $C_{G,X}^{H,Y}$ so defined is called the \emph{$H$-discriminating complexity of $G$ with respect to finite generating sets $X$ and $Y$.}
\end{defn}

We will be interested in asymptotic classes of the discriminating complexity for a given group.
To this end, if $f,g : \mathbb{N} \rightarrow \mathbb{N}$, we say that $f$ is \emph{asymptotically dominated by} $g$, denoted $f \preceq g$, if there is a constant $K$ such that for all $n$,
\begin{equation*}
  f(R) \leq Kg(KR)+K.
\end{equation*}
We say that $f$ is \emph{asymptotically equivalent} to $g$, denoted $f \approx g$, if $f \preceq g$ and $g \preceq f$.

Lemma \ref{Solie--Lemma--Change of generating set} and the remark following it imply the following proposition.

\begin{prop}
  Let $G$ be a fully residually $H$ group.
  Let $X, X'$ be finite generating sets for $G$, and let $Y, Y'$ be finite generating sets for $H$.
  Then we have
  \begin{displaymath}
    C_{G,X}^{H,Y} \preceq C_{G,X'}^{H,Y'}.
  \end{displaymath}
\end{prop}

As a result of the above proposition, the asymptotic class of the $H$-discriminating complexity of $G$ is invariant with respect to choice of finite generating set for both $G$ and $H$.
Therefore, we will omit reference to these generating sets and simply indicate (the asymptotic class of) the $H$-discriminating complexity of $G$ by $C_G^H$.

In order to study $H$-discriminating complexity, we will find it useful to establish some notation for sequences of homomorphisms which discriminate larger and larger balls in a given group.

\begin{defn}[Discriminating sequence] \index{discriminating sequence}
  \label{Defn--Discriminating sequence}
  Let $\Phi = (\phi_R: G \rightarrow H)_{R \in \mathbb{N}}$ be a sequence of homomorphisms.
  If for each $R \in \mathbb{N}$, the set $B_R(G,X)-1$ is $H$-discriminated by $\phi_R$, we say that $\Phi$ is a \emph{$H$-discriminating sequence} with respect to the finite generating set $X$.
\end{defn}

  It is straightforward to see that a finitely generated group $G$ is fully residually $H$ if and only if $G$ admits an $H$-discriminating sequence with respect to some (every) finite generating set.

  We also make the following observation. Let $X$ and $X'$ be finite generating sets for $G$ and let $\Phi$ be an $H$-discriminating sequence for $G$ with respect to $X$.
  By passing to an arithmetic subsequence of $\Phi$, we may obtain an $H$-discriminating sequence with respect to $X'$, and the complexity of this subsequence is equivalent to that of $\Phi$.

\begin{defn}[Complexity function] \index{complexity function}
  \label{Defn--Complexity function}
  Given an $H$-discriminating sequence $\Phi$, we construct the \emph{$H$-discriminating complexity function associated to $\Phi$,} the function $C_\Phi: \mathbb{N} \rightarrow \mathbb{N}$ defined via:
  \begin{displaymath}
    C_\Phi(R) := |\phi_R|_X^Y.
  \end{displaymath}
\end{defn}

We briefly note that complexity functions of discriminating sequences provide an obvious upper bound for discriminating complexity.

\begin{prop}
  Let $G$ and $H$ be finitely generated groups and let $G$ be fully residually $H$.
  Let $\Phi = (\phi_R)_{R \in \mathbb{N}}$ be an $H$-discriminating sequence for $G$.
  Then $C_G^H \preceq C_\Phi$.
\end{prop}

\subsubsection{Free Abelian groups}

We begin by investigating the $\mathbb{Z}$-discriminating complexity of a free Abelian group $\mathbb{Z}^n$.

\begin{prop}
  \label{Solie--Prop--Upper Bound on Discriminating for Free Abelian Groups}
  The $\mathbb{Z}$-discriminating complexity of $\mathbb{Z}^n$ is asymptotically dominated by a polynomial of degree $n-1$.
\end{prop}

We will consider the elements of $\mathbb{Z}^n$ to be $n$-tuples of integers.
For $R \in \mathbb{N}$, define $[-R,R]^n := \{ (t_1, \dots, t_n) \in \mathbb{Z}^n : |t_i| \leq R,\; 1 \leq i \leq n \}$.
Instead of discriminating closed balls in $\mathbb{Z}^n$ with respect to the usual metric, we will construct homomorphisms which are injective on the sets $[-R,R]^n$ for each $R \in \mathbb{N}$.

\begin{lem}
  \label{Solie--Lemma--Discriminating Free Abelian Map}
  For $n,R \in \mathbb{N}$, define the homomorphism $\theta_{n,R}: \mathbb{Z}^n \rightarrow \mathbb{Z}$ by
  \begin{align*}
    \theta_{n,R}(t_1, \dots, t_n) = \displaystyle{\sum_{i=1}^n (2R+1)^{i-1} t_i}.
  \end{align*}

  Then $\theta_{n,R}$ induces a bijection from $[-R,R]^n$ to the interval
  \begin{displaymath}
    I_{n,R}:=\left[-\frac{1}{2}\left((2R+1)^n-1\right), \frac{1}{2}\left((2R+1)^n-1\right)\right].
  \end{displaymath}
\end{lem}

\begin{proof}
  We proceed by induction.
  Since $\theta_{1,R}$ is the identity for all $R$, we have the promised bijection for $n=1$.

  Fix $r$ and assume that $\theta_{n,R}$ induces a bijection from $[-R,R]^n$ to $I_{n,R}$.
  Note that that
  \begin{displaymath}
    \theta_{n+1,R}(t_1, \dots, t_{n+1}) = \theta_{n,R}(t_1, \dots, t_n)+(2R+1)^n t_{n+1}.
  \end{displaymath}

  By the inductive hypothesis, we have
  \begin{align*}
    \big|\theta_{n+1,R}(t_1, \dots, t_{n+1})\big| &\leq \big|\theta_{n,R}(t_1, \dots, t_n)\big| + (2R+1)^n \big|t_{n+1}\big| \\
    & \leq \frac{1}{2}\big((2R+1)^n-1\big) + R(2R+1)^n \\
    & = \frac{1}{2}(2R+1)^n + \frac{1}{2}2R(2R+1)^n - \frac{1}{2} \\
    & = \frac{1}{2}\big((2R+1)^{n+1} - 1\big).
  \end{align*}

  Therefore $\theta_{n+1,R}$ maps $[-R,R]^{n+1}$ into the interval $I_{n+1,R}$.

  Suppose that there are $(s_1, \dots, s_n), (t_1, \dots, t_n) \in [-R,R]^{n+1}$ such that $\theta_{n+1,R}(t) = \theta_{n+1,R}(s)$.
  We then have
  \begin{displaymath}
    \theta_{n,R}(t_1, \dots, t_n)+(2R+1)^n t_{n+1} = \theta_{n,R}(s_1, \dots, s_n)+(2R+1)^n s_{n+1}.
  \end{displaymath}
  We must have $t_{n+1} \neq s_{n+1}$ or we contradict the injectivity of $\theta_{n,r}$.
  However, by using the inductive hypothesis, we have
  \begin{align*}
    (2R+1)^n-1 &\geq \big|\theta_{n,R}(t_1, \dots, t_n) - \theta_{n,R}(s_1, \dots, s_n)\big| \\
      &= \big|(2R+1)^n(s_{n+1}-t_{n+1})\big| \\
      &\geq (2R+1)^n,
  \end{align*}
  a contradiction.

  We have shown that $\theta_{n+1,R}$ maps $[-R,R]^{n+1}$ injectively to $I_{n+1,R}$.
  Since both sets have the same cardinality, $\theta_{n+1,R}$ is a bijection between $[-R,R]^{n+1}$ and $I_{n+1,R}$.
\end{proof}

Proposition \ref{Solie--Prop--Upper Bound on Discriminating for Free Abelian Groups} follows immediately from Lemma \ref{Solie--Lemma--Discriminating Free Abelian Map} since each homomorphism $\theta_{n,R}$ is injective on $B_R$ and therefore discriminates $B_R-1$.
Furthermore, the complexity of $\theta_{n,R}$ is $(2R+1)^{n-1}$, as promised.

The following result is well-known from number theory and will help us to establish a lower bound on the $\mathbb{Z}$-discriminating complexity of $\mathbb{Z}^n$.

\begin{siegelslemma}[\cite{Bombieri1984,Siegel1929}]
  \label{Siegel--Prop--Siegel's Lemma}
  Let $A$ be an $M \times N$ integer matrix with $M > N$ and $A \neq 0$.
  Let $B$ be a constant such that for every entry $a_{ij}$ of $A$, we have $|a_{ij}| \leq B$.
  Then there exists a nonzero $N \times 1$ integer matrix $X$ with entries $x_i$ such that $AX = 0$ and for each $i$,
  \begin{equation*}
    |x_i| \leq (NB)^{M/(N-M)}.
  \end{equation*}
\end{siegelslemma}


\begin{cor}
  \label{Solie--Corollary--Lower Bound on Discriminating for Free Abelian Groups}
  The $\mathbb{Z}$-discriminating complexity of $\mathbb{Z}^n$ asymptotically dominates a polynomial of degree $n-1$.
\end{cor}

\begin{proof}
  Let $\Phi = (\phi_R)_{R \in \mathbb{N}}$ be a $\mathbb{Z}$-discriminating sequence for $\mathbb{Z}^n$.
  By definition, $\phi_R$ discriminates the set $B_R-1$, the closed ball of radius $R$ with respect to (WLOG) the standard basis of $\mathbb{Z}^n$.

  Each $\phi_R$ can be represented by an $n \times 1$ integer matrix whose entries are bounded above in magnitude by $C_\Phi(R)$.
  By Siegel's lemma, there exists for each $\phi_R$ an element of the kernel of $\phi_R$ whose entries are bounded above in magnitude by $(nC_\Phi(R))^{1/(n-1)}$.
  Since $\phi_R$ discriminates $B_R-1$, it also discriminates the set of nontrivial elements whose entries are bounded above in magnitude by $\lfloor R/n \rfloor$.
  We must then have
  \begin{align*}
    \dfrac{R}{n} - 1 \leq \displaystyle{\left\lfloor \dfrac{R}{n} \right\rfloor} &\leq (nC_\Phi(R))^{{1}/{(n-1)}}\\
    \dfrac{(R-n)^{n-1}}{n^{n-1}} &\leq nC_\Phi(R)\\
    \dfrac{(R-n)^{n-1}}{n^n} & \leq C_\Phi(R).
  \end{align*}
  Therefore $C_\Phi(R) \succeq R^{n-1}$.

  In particular, taking $\Phi$ such that $C_\Phi(R) = C_G^\Gamma(R)$, we have that $C_G^\Gamma(R) \succeq R^{n-1}$.
\end{proof}

\begin{thm}
  \label{Solie--Theorem--Free Abelian rank n has complexity n-1}
  The $\mathbb{Z}$-discriminating complexity of $\mathbb{Z}^n$ is asymptotically equivalent to a polynomial of rank $n-1$.
\end{thm}

For $p \in \mathbb{Z}$, define a homomorphism $\theta_{n,R}^{p}: \mathbb{Z}^n \rightarrow \mathbb{Z}$ by
\begin{equation*}
  \theta_{n,R}^{p}(t_1, \dots, t_n) := p \theta_{n,R}(t_1, \dots, t_n).
\end{equation*}
Note that since $\theta_{n,R}$ discriminates the set $[-R,R]^n-1$, if $i \in \theta_{n,R}^{p}([-R,R]^n-1)$, then $|i| > |p|$.
Clearly $\theta_{n,R}^{p}$ then also discriminates $[-R,R]^n-1$.

\subsubsection{Extensions of centralizers}

Let $\Gamma$ be a non-Abelian, torsion-free hyperbolic group.
Let $G$ be an iterated extension of centralizers over $\Gamma$ with finite generating set $X$, and let $u \in G$ be a hyperbolic element which generates its own centralizer.
Let $G'$ be a rank $n$ extension of the centralizer $C(u) = C_G(u)$.
Fix elements $T = \{t_1, \dots, t_n\} \subset G'$ be such that $\{u, t_1, \dots, t_n\}$ is a basis for the free Abelian group $C_{G'}(u)$.

We define a homomorphism $\Theta_{n,R}^{p} : G' \rightarrow G$ via:
\begin{align*}
  \Theta_{n,R}^{p}(g) &:= g \text{ for all } g \in G \\
  \Theta_{n,R}^{p}(t_i) &:= u^{p(2R+1)^{i-1}} \text{ for } i=1, \dots, n.
\end{align*}

By putting $T$ in bijection with the standard basis for $\mathbb{Z}^n$, it is clear that the homomorphism $\Theta_{n,R}^{p} \mid_{\langle T \rangle}$ is equivalent to $\theta_{n,R}^{p}$.
Consequently, for all nontrivial $a \in \langle T \rangle$ is such that $|a|_T < R$, then $\Theta_{n,R}^{p}(a)$ is a power of $u$ of exponent greater than or equal to $p$ in magnitude.
We further observe that $\Theta_{n,R}^{p}$ is a retraction onto $G$.

\begin{lem}
  Let $w$ be an element of $G'$ with $|w|_{X \cup T} \leq R$.
  There is a linear function $N_3: \mathbb{N} \rightarrow \mathbb{N}$ such that $\Theta_{n,R}^{N_3(R)}(w) \neq 1$.
\end{lem}

\begin{proof}
  Since $G'$ is an amalgamated product, we may write $w$ as a geodesic $X \cup T$-word
  \begin{equation}
    w = g_0 a_0 g_1 a_1 \cdots g_k a_k g_{k+1}
  \end{equation}
  where for each $i$, $g_i$ is an $X$-word and $a_i$ is a $T$-word.
  We may further assume that no $g_i$ or $a_i$ is the empty word, except possibly $g_0$, $g_{k+1}$, or both.

  First, we may assume that if some $g_i$ is not a power of $u$, then no $g_i$ is a power of $u$.
  To see this, suppose that $g_j$ is some power of $u$ but $g_{j-1}$ is not, and consider the subword $g_{j-1} a_{j-1} g_j a_j$.
  Since $a_{j-1}$ is a word in the generators $T$, it represents an element of the centralizer of $u$.
  Consequently, we may rewrite this subword as $g_{j-1} g_j a_{j-1} a_j$ without increasing the $X \cup T$-length of the overall word.
  By replacing $g_{j-1} g_j$ and $a_{j-1} a_j$ with possibly shorter words representing the same elements, we obtain another word representing $w$ in $G'$ of length at most $R$.

  Define
  \begin{equation*}
    N_3(R) := N_2(R) + R + 1
  \end{equation*}
  and note that, because $N_2(R)$ is linear in $R$, the function $N_3(R)$ is also linear in $R$ .

  Consider the homomorphism $\Theta_{n,R}^{N_3(R)}: G' \rightarrow G$.
  Then
  \begin{align*}
    \Theta_{n,R}^{N_3(R)}(w) &= g_0 u^{r_0} g_1 u^{r_1} g_2 u^{r_2} \cdots g_k u^{r_k} g_{k+1}\\
      &= g_0 w_{\mathbf{g}}(\mathbf{r}) g_{k+1},
  \end{align*}
  where $\mathbf{g} = (g_1, \dots, g_k)$, $\mathbf{r} = (r_0, \dots, r_k)$, $w_{\mathbf{g}}(\mathbf{r})$ is as in Equation \ref{Eqn--Padded general multiword} possibly $g_0$ or $g_{k+1}$ or both are trivial.
  Since $|a_i|_{T} \leq |w|_{X \cup T} \leq R$, we have $\min(\mathbf{r}) > N_2(R)$ for all $i$.
  Since $\sum |g_i|_{X} \leq R$ and $G$ is relatively hyperbolic with $u$ a hyperbolic element generating its own centralizer, by Theorem \ref{Solie--Lemma--Padding for more general words} we have that $\Theta_{n,R}^{N_3(R)}(w) \neq 1$ in $G$.

  Now suppose that $w$ can be written as a geodesic $(X \cup T)$-word
  \begin{equation*}
    w = u^{r_0} a_0,
  \end{equation*}
  where $r_0$ is an integer, $a_0$ is a nonempty $T$-word, and $|u^{r_0}|_X + |a_0|_T \leq R$.
  Since $|u|_X \geq 1$, $|r_0| \leq R$.
  By definition, $\Theta_{n,R}^{N_3(R)}(a) = u^{e}$ where $|e| > R$, and so $\Theta_{n,R}^{N_3(R)}(w) \neq 1$ in $G$.
\end{proof}

\begin{thm}
  \label{Solie--Theorem--EOC has poly upper}
  Let $G$ be an iterated extension of centralizers over $\Gamma$.
  Let $G'$ be a rank $n$ extension of a cyclic centralizer of $G$.
  Then the $G$-discriminating complexity of $G'$ is asymptotically dominated by a polynomial of degree $n$.
\end{thm}

\begin{proof}
  By the previous theorem, the homomorphism $\Theta_{n,R}^{N_3(R)}$ maps all elements of $G'$ with $X \cup T$-length at most $R$ to nontrivial elements of $G$. Therefore, $\left(\Theta_{n,R}^{N_3(R)}\right)_{R\in \mathbb{N}}$ is a $G$-discriminating sequence for $G'$.

  To compute the complexity of $\Theta_{n,R}^{N_3(R)}$, we first note that $\Theta_{n,R}^{N_3(R)}$ fixes elements of $X$.
  For $t_i \in T$, we have $\Theta_{n,R}^{N_3(R)}(t_i) = u^{(N_3(R))(2R+1)^{i-1}}$.
  Therefore, as a function of $r$,
  \begin{equation*}
    |\Theta_{n,R}^{N_3(R)}| \leq |u|_X (N_3(R))(2R+1)^{n-1} \approx R^n,
  \end{equation*}
  since $N_2(R)$ is linear in $R$.
  Thus the complexity of the sequence $\left(\Theta_{n,R}^{N_3(R)}\right)_{R\in \mathbb{N}}$ is asymptotically dominated by $R^n$.
\end{proof}

\subsubsection{Iterated extensions of centralizers}

\begin{thm}
  \label{Solie--Theorem--IEOC has poly upper}
  The $\Gamma$-discriminating complexity of an iterated extension of centralizers over $\Gamma$ is asymptotically dominated by a polynomial with degree equal to the product of the ranks of the extensions.
\end{thm}

\begin{proof}
  Let $G$ be an iterated extension of centralizers over $\Gamma$, and let
  \begin{equation*}
    \Gamma = G_0 \leq G_1 \leq \dots \leq G_k = G
  \end{equation*}
  be a sequence such that $G_{i}$ is an extension of a centralizer of $G_{i-1}$ for $i=1, \dots, k$.

  By Theorem \ref{Solie--Theorem--EOC has poly upper}, each $G_i$ has a $G_{i-1}$-discriminating family with complexity polynomial of degree equal to the rank of the extension.
  By composing these families, we obtain a $\Gamma$-discriminating sequence for $G$ which is also of polynomial complexity; in particular, the properties of complexity imply that the degree of the polynomial is equal to the product of the ranks of the extensions required to construct $G$.
\end{proof}

\subsubsection{Arbitrary $\Gamma$-limit groups}

\begin{thm}
  \label{Solie--Theorem--limit group has poly upper}
  The $\Gamma$-discriminating complexity of any $\Gamma$-limit group is asymptotically dominated by a polynomial.
\end{thm}

\begin{proof}
  Let $G$ be a $\Gamma$-limit group.
  By Proposition \ref{Kharlampovich-Myasnikov--Prop--Embedding theorem for limit groups}, there is a $G'$ which is an iterated extension of centralizers over $\Gamma$ such that $G \leq G'$.
  Choose a finite generating set $X$ for $G'$ which includes a finite generating set $Y$ for $G$.
  Then for all $R \in \mathbb{N}$, we have $B_R(G,Y) \subseteq B_R(G',X)$, so a $\Gamma$-discriminating sequence exists for $G'$ which is also a $\Gamma$-discriminating sequence for $G$.
\end{proof}

\begin{lem}
  \label{Solie--Lemma--EOC has poly lower}
  Let $G$ be a $\Gamma$-limit group with a free Abelian subgroup of rank $n+1$.
  Then the $\Gamma$-discriminating complexity of $G$ asymptotically dominates a polynomial of degree $n$.
\end{lem}

\begin{proof}
  Since the asymptotic class of the complexity of a $\Gamma$-discriminating sequence is invariant with respect to choice of finite generating set, we may choose a generating set $Y$ for $G$ with a subset $T \subseteq Y$ such that $\langle T \rangle$ is free Abelian of rank $n+1$.
  Let $\Phi = (\phi_R)$ be a $\Gamma$-discriminating sequence for $G$.
  Since $\Gamma$ is torsion-free hyperbolic, every Abelian subgroup of $\Gamma$ is isomorphic to $\mathbb{Z}$, and therefore every $\phi_R$ must map $\langle T \rangle$ to a cyclic subgroup.
  Since $T \subseteq Y$, restricting $\Phi$ to $\langle T \rangle$ gives us a $\mathbb{Z}$-discriminating sequence for $\langle T \rangle \cong \mathbb{Z}^{n+1}$.
  Therefore, the complexity of $\Phi$ must asymptotically dominate a polynomial of degree $n$ by Proposition \ref{Solie--Corollary--Lower Bound on Discriminating for Free Abelian Groups}.
\end{proof}



\bibliographystyle{plain}
\bibliography{LimitBib}

\end{document}